\documentclass[11pt]{amsproc}

\usepackage[hiresbb]{graphicx, xcolor}
\usepackage{wrapfig}
\usepackage{bm}

\usepackage{amscd, amsmath, amsthm, amssymb, mathdots, mathtools, mathrsfs, stmaryrd}
\usepackage{ascmac}
\usepackage{comment}
\usepackage{bm, bbm}
\allowdisplaybreaks

\usepackage{CJKutf8}

\usepackage{fullpage}

\usepackage{hyperref}

\usepackage{tikz}
\usetikzlibrary{intersections, calc, arrows.meta}

\usepackage{here}
\usepackage{time}
\usepackage[abbrev]{amsrefs}

\usepackage{xcolor}
\usepackage[capitalize,nameinlink,noabbrev,nosort]{cleveref}
\hypersetup{
	colorlinks=true,       
	linkcolor=brown,          
	citecolor=brown,        
	filecolor=brown,      
	urlcolor=brown,           
}

\makeatletter
\@namedef{subjclassname@2020}{\textup{2020} Mathematics Subject Classification}
\makeatother

\newtheorem*{theorem*}{\hspace{-6.3mm}\textbf{Theorem}}  

\newtheorem{theoremcounter}{Theorem Counter}[section]

\theoremstyle{remark}
\newtheorem{remark}[theoremcounter]{Remark}

\theoremstyle{definition}
\newtheorem{definition}[theoremcounter]{Definition}
\newtheorem{example}[theoremcounter]{Example}

\theoremstyle{plain}
\newtheorem{lemma}[theoremcounter]{Lemma}
\newtheorem{proposition}[theoremcounter]{Proposition}
\newtheorem{corollary}[theoremcounter]{Corollary}

\newtheorem{theorem}[theoremcounter]{Theorem}

\numberwithin{equation}{section}

\newcommand{\Z}{\mathbb{Z}}
\newcommand{\Q}{\mathbb{Q}}
\newcommand{\R}{\mathbb{R}}
\newcommand{\C}{\mathbb{C}}

\newcommand{\bbH}{\mathbb{H}}

\newcommand{\calQ}{\mathcal{Q}}

\DeclareMathOperator{\ImNew}{Im}
\renewcommand{\Im}{\ImNew}

\DeclareMathOperator{\SL}{SL}

\DeclareMathOperator{\GL}{GL}

\newcommand{\pmat}[1]{\begin{pmatrix}#1\end{pmatrix}}

\newcommand{\smat}[1]{\bigl(\begin{smallmatrix}#1\end{smallmatrix}\bigr)}

\def\piros#1{{\textbf{\color{red}#1}}}%
\begin{document}

\title[]{On the Alexander polynomials of modular knots}

\author[]{Soon-Yi Kang}
\address{Department of Mathematics, Kangwon National University, Chuncheon, 24341, Republic of Korea}
\email{sy2kang@kangwon.ac.kr}

\author[]{Toshiki Matsusaka}
\address{Faculty of Mathematics, Kyushu University, Motooka 744, Nishi-ku, Fukuoka 819-0395, Japan}
\email{matsusaka@math.kyushu-u.ac.jp}

\author[]{Kyungbae Park}
\address{Department of Mathematics, Kangwon National University, Chuncheon, 24341, Republic of Korea}
\email{kyungbaepark@kangwon.ac.kr}





\maketitle

\begin{abstract}
	Closed geodesics associated with indefinite binary quadratic forms, or equivalently with real quadratic irrationals, have long been studied as geometric $\SL_2(\Z)$-invariants. Building on the Birman--Williams approach to Lorenz knots and following the notion of modular knots introduced by Ghys, this article investigates the topological $\SL_2(\Z)$-invariants arising from modular knots. Our main focus is the Alexander polynomial of modular knots. Using the Burau representation, we highlight two contrasting features of this family. On the one hand, for each fixed degree, only finitely many Alexander polynomials of modular knots occur. On the other hand, any integer appears as a coefficient of the Alexander polynomial of some modular knot, and coefficients of the same sign can occur in runs of arbitrarily long length.
\end{abstract}


\section{Introduction}

Since the time of Gauss, invariants of binary quadratic forms under $\SL_2(\Z)$-action have attracted considerable attention. Let 
\[
	\mathcal{Q} = \{Q(x,y) = ax^2 + bxy + cy^2 : a,b,c \in \Z\}
\]
be the set of all integral binary quadratic forms. The group $\SL_2(\Z)$ acts on $\mathcal{Q}$ by
\[
	\left(Q \circ \pmat{\alpha & \beta \\ \gamma & \delta} \right)(x,y) = Q(\alpha x+\beta y, \gamma x+\delta y).
\]
Since the discriminant $\mathrm{disc}(Q) = b^2 - 4ac$ is invariant under this action, for each fixed integer $D \equiv 0, 1 \pmod{4}$ we consider the induced action on the subset $\mathcal{Q}_D \coloneqq \{Q \in \mathcal{Q} : \mathrm{disc}(Q) = D\}$. As Gauss first observed, for a non-square $D$, the quotient $\mathcal{Q}_D/\SL_2(\Z)$ is finite, and by considering the composition of binary quadratic forms, he revealed the underlying group-theoretic structure,
(see Halter-Koch's book~\cite{HalterKoch2013} for details). 

More generally, a map from $\calQ/\SL_2(\Z)$ to a set $X$ is called a \emph{class invariant}. For instance, the discriminant function 
\[
	\calQ/\SL_2(\Z) \to \Z; [Q] \mapsto \mathrm{disc}(Q)
\]
is one such invariant. Here we further restrict our attention to the case of positive quadratic discriminant, namely,
\[
	\calQ_+ = \{Q \in \calQ : \mathrm{disc}(Q) > 0 \text{ is a non-square}\} = \bigcup_{0 < D \neq \square} \calQ_D.
\]
Another classical example is the \emph{closed geodesic} determined by $Q \in \calQ_+$. To be more precise, let $w, w'$ be the two real quadratic irrational roots of $Q(z,1) = 0$, which are Galois conjugate. Consider the geodesic in the upper half-plane $\bbH = \{z \in \C : \Im(z) > 0\}$ equipped with the usual hyperbolic metric, connecting $w'$ and $w$. The group $\SL_2(\Z)$ acts on $\bbH$ by the M\"{o}bius transformations, and the above geodesic projects to a closed geodesic $C_Q$ on the quotient space $\SL_2(\Z) \backslash \bbH$. Then, it is well known that the length of the closed geodesic 
\[
	\calQ_+/\SL_2(\Z) \to \R; [Q] \mapsto \mathrm{length}(C_Q)
\]
can be expressed in terms of the fundamental unit of the real quadratic field $\Q(w)$. Other \emph{geometric} class invariants determined by such geodesics have also been studied more recently by Duke--Imamo\={g}lu--T\'{o}th~\cite{DukeImamogluToth2016}.

In this article, we aim to lift geodesics on $\SL_2(\Z) \backslash \bbH$ to knots in the three-dimensional space $\SL_2(\Z) \backslash \SL_2(\R)$ and study their \emph{topological invariants}. This approach builds on the work of Ghys~\cite{Ghys2006}. He studied the geodesic flow on $\SL_2(\Z) \backslash \SL_2(\R)$, a space known to be homeomorphic to the complement of the trefoil knot $K_{2,3}$ in the three-dimensional sphere $S^3$, (see~\cite[p.84]{Milnor1971}). The closed orbits of this flow are called \emph{modular knots}, and under the projection $\SL_2(\Z) \backslash \SL_2(\R) \to \SL_2(\Z) \backslash \bbH; g \mapsto gi$, they project to the closed geodesics discussed above. Ghys showed that each closed orbit can be described using the symbolic dynamics introduced in the study of Lorenz knots. We adopt this braid description as the definition of modular knots and present it in detail in \cref{sec:modular-knot}. See also the visual introduction by Ghys and Leys~\cite{GhysLeys}. 

As a first example, Ghys himself considered a topological invariant of links, the \emph{linking number}, for the link formed by a single modular knot $\widehat{\mathbb{M}}(Q)$ associated with a class $[Q]$ and the missing trefoil $K_{2,3}$,
\[
	\calQ_+/\SL_2(\Z) \to \Z; [Q] \mapsto \mathrm{link}(\widehat{\mathbb{M}}(Q), K_{2,3}).
\]
Then he showed that it coincides with the \emph{Rademacher symbol}, which appears as a special value of the partial zeta function associated with the class $[Q]$, (see Zagier~\cite[$\S$ 14, Satz 2]{Zagier1981}).

Even when restricted to the study of linking numbers, research has developed in remarkably diverse directions. First, analogues of the prime geodesic theorem that incorporate both the linking number and the length of closed geodesics were studied by Sarnak~\cite{Sarnak2010}, Mozzochi~\cite{Mozzochi2013}, and Kelmer~\cite{Kelmer2012}. More recently, building on Burrin's extension of the Rademacher symbol~\cite{Burrin2022}, Burrin and von Essen~\cite{BurrinEssen2024} further generalized these results to $\Gamma \backslash \SL_2(\R)$ for more general Fuchsian groups $\Gamma$, not necessarily $\SL_2(\Z)$. It is worth noting that, in extending from $\SL_2(\Z)$ to a general $\Gamma$, the Rademacher symbol is not generalized through its description in terms of the partial zeta function mentioned above. Instead, the extension is constructed on the basis of Rademacher's original definition~\cite{Rademacher1956, Rademacher1972} as a class invariant function on $\SL_2(\Z)$ together with several equivalent reformulations summarized by Atiyah~\cite{Atiyah1987}.

Second, in the case where $\Gamma$ is a triangle group, Dehornoy and Pinsky~\cite{Pinsky2014, Dehornoy2015-2, DehornoyPinsky2018} extended the notion of modular knots using the theory of templates, and studied the linking numbers between the modular-like knot and the missing knot intrinsic to the space $\Gamma \backslash \SL_2(\R)$. The arithmetic aspects of the Rademacher symbol for triangle groups have been studied by the second author together with Ueki~\cite{MatsusakaUeki2023} and with Shin~\cite{MatsusakaShin2025}.

Finally, returning to the original setting of $\SL_2(\Z)$, Simon~\cite{Simon2025}, in response to a question raised by Ghys~\cite{Ghys2006}, derived several formulas for the linking number not with the missing trefoil but between two modular knots. From the viewpoint of modular forms, this topic has also been investigated by Duke--Imamo\={g}lu--T\'{o}th~\cite{DukeImamogluToth2017} and the second author~\cite{Matsusaka2024}. Although we do not go into details here, the latter approaches require a suitable modification of the definition of modular knots to avoid linking with the missing trefoil.

In this article, as another topological invariant, we consider the \emph{Alexander polynomial} $\Delta_K(t)$ of a modular knot $K = \widehat{\mathbb{M}}(Q)$, that is,
\[
	\calQ_+/\SL_2(\Z) \to \Z[t, t^{-1}]; [Q] \mapsto \Delta_{\widehat{\mathbb{M}}(Q)}(t).
\]
Here, we use a normalized version so that, for modular knots, it becomes a polynomial in $t$. In \cref{sec:modular-knot} and \cref{sec:Lorenz-braid}, we first give a precise definition of a modular knot. Here, instead of using the geodesic flow, we adopt an equivalent formulation observed by Ghys~\cite{Ghys2006}, in which a real quadratic irrational $w$ is used to define a modular braid from the coefficients of its continued fraction expansion, and the modular knot is then obtained as the closure of this braid, (see also the surveys by Dehornoy~\cite[Section 4]{Dehornoy2011} and Birman~\cite{Birman2013}). 
Building on the work of Birman--Williams~\cite{BirmanWilliams1983} on Lorenz knots, $(p,q)$-torus knots can be realized as modular knots for coprime positive integers $p$ and $q$. In \cref{sec:Torus-knot}, we determine explicitly which real quadratic irrational numbers (or indefinite binary quadratic forms) correspond to a given $(p,q)$-torus knot (\cref{thm:Torus-knot-modular}), by using combinatorial notions such as snake graphs and Christoffel words. In \cref{sec:topological}, we summarize several known results on knot invariants rewritten in the language of modular knots.

We then turn to the Alexander polynomials. In \cref{sec:Alexander}, we review its definition and provide a formula for modular knots (\cref{thm:Burau-explicit}) via the Burau representation. This formula allows for a variety of observations through numerical experiments. In particular, unlike the linking number, it is expected to distinguish modular knots more precisely and to extract more detailed information. In the following two sections, we present several results characterizing families of modular knots that have been discovered through these numerical experiments.

In \cref{sec:Alexander-finite}, although this fact is already known to experts studying Lorenz knots, we show that for each $n$, there are only finitely many Alexander polynomials (up to multiplication by unit elements in $\Z[t,t^{-1}]$) of modular knots of degree $n$. 

\begin{theorem}[\cref{thm:fin-Alexander}]\label{thm-1}
	For a positive integer $n$, we define
	\[	
		A_n = \{\Delta_K(t) : \deg \Delta_K = n, K: \text{modular knot}\} \subset \Z[t].
	\]
	Then we have $A_n = \emptyset$ for every odd $n$, and $\# A_n \le p(n)$ for every even $n$, where $p(n)$ denotes the number of partitions of $n$.
\end{theorem}

Note that any (resp.~monic) reciprocal polynomial taking the value $1$ at $t=1$ can be realized as the Alexander polynomial of some (resp.~fibered) knot, due to Seifert~\cite{Seifert1935} and Burde~\cite{Burde1966}, respectively. In comparison with these facts, \cref{thm-1} indicates that the family of modular knots is quite small within the set of all (resp.~fibered) knots.

On the other hand, in \cref{sec:Coeff-Alexander}, we show that the family of modular knots is relatively large from two perspectives. For instance, the modular knot $\widehat{\mathbb{M}}(Q)$ associated with $Q(x,y) = 152x^2 - 600xy - 237y^2$ has the braid expression given in \cref{Lorenz-braid-example}, which is repeatedly used as an example throughout this article. Its Alexander polynomial is computed in \cref{ex:Alex-Lorenz} as
\[
	\Delta_{\widehat{\mathbb{M}}(Q)}(t) = t^{20}-t^{19}+t^{17}-t^{16}+t^{14}-t^{13}+t^{12}-t^{11}+t^{10}-t^9+t^8-t^7+t^6-t^4+t^3-t+1.
\]
In many computed examples of Alexander polynomials of modular knots, the coefficients lie in $\{-1, 0, 1\}$ and alternate in sign, a phenomenon that occurs for torus knots and, more generally, for $L$-space knots~\cite{OzsvathSzabo2005}. However, the following theorems show that this behavior does not persist in general.

\begin{theorem}[\cref{thm:coeff-Alexander}]\label{thm-2}
	Every integer appears as a coefficient of the Alexander polynomial of a modular knot.
\end{theorem}

\begin{theorem}[\cref{thm:neg-long}]\label{thm-3}
	For the coefficients of the Alexander polynomial of a modular knot, negative coefficients appear consecutively for an arbitrarily long run.
\end{theorem}

\cref{thm-2} serves as an analogue of Schur and Suzuki's theorem~\cite{Suzuki1987} for cyclotomic polynomials. \cref{thm-2} and \cref{thm-3} contrast sharply with the situation for torus knots and $L$-space knots. This indicates that modular knots form a considerably large and more varied class. 
Finally, although we do not discuss it in detail here, we mention a few related results. Dehornoy~\cite{Dehornoy2015} investigated the Alexander polynomial of Lorenz knots (which coincide with modular knots) and showed that their zeros concentrate on a certain annulus. This result characterizes the family of Lorenz knots within the set of all knots. Dehornoy~\cite[Section 4]{Dehornoy2011} also examined Gauss' composition and gave several observations on when a modular knot becomes the trivial knot.

\section*{Acknowledgements}

The first author was supported by the National Research Foundation of Korea (NRF) funded by the Ministry of Education (NRF-2022R1A2C1007188),
the second author was supported by JSPS KAKENHI (JP21K18141 and JP24K16901), and 
the third author was supported by the NRF grant funded by the Korea government
(MSIT) (RS-2025-24523511).
In addition, all three authors were supported by a research project carried out by KRIMS, supported by the NRF funded by the Ministry of Education (RS-2025-25415913).

\section{Modular braids and modular knots} \label{sec:modular-knot}

In the following two sections, building on Ghys' idea in \cite[Section 3.5]{Ghys2006}, we introduce modular knots within the framework of the Birman--Williams description of Lorenz knots via symbolic dynamics~\cite{BirmanWilliams1983}. Specifically, in this section, we define the modular braids associated with a binary quadratic form class $[Q]$ (\cref{def:modular-knot-Ber}) and their reformulation in terms of words (\cref{def:modular-knot-word}), and we show in \cref{prop:two-modular-braids} that these definitions are equivalent.

\subsection{Braid groups}

To begin with, we briefly review the braid group, which forms the foundation for the definitions of modular knots and the Alexander polynomials. For details, see Kassel--Turaev~\cite{KasselTuraev2008} for instance.

\begin{definition}
	For a positive integer $n$, the \emph{Braid group} $B_n$ is the group generated by $\sigma_1, \sigma_2, \dots, \sigma_{n-1}$ satisfying the \emph{braid relations}
	\[
		\sigma_i \sigma_j = \sigma_j \sigma_i
	\]
	for all $1 \le i, j \le n-1$ with $|i-j| \ge 2$, and 
	\[
		\sigma_i \sigma_{i+1} \sigma_i = \sigma_{i+1} \sigma_i \sigma_{i+1}
	\]
	for all $1 \le i \le n-2$. An element $\sigma \in B_n$ is called a \emph{braid}.
\end{definition}

An element $\sigma \in B_n$ can be represented by a braid diagram with $n$ strands, where each generator $\sigma_i$ denotes a crossing between the $i$-th and $(i+1)$-th strands.

\begin{figure}[H]
	\centering
	\begin{tikzpicture}[line width=1pt, scale=0.4]
	\begin{scope}[shift={(0,0)}]
		\draw (0,0) -- (12,0);
		\draw (0,4) -- (12,4);

		\def \dotPt {4pt}; 

		\fill (1,0) circle (\dotPt);
		\fill (3,0) circle (\dotPt);
		\fill (5,0) circle (\dotPt);
		\fill (7,0) circle (\dotPt);
		\fill (9,0) circle (\dotPt);
		\fill (11,0) circle (\dotPt);

		\fill (1,4) circle (\dotPt);
		\fill (3,4) circle (\dotPt);
		\fill (5,4) circle (\dotPt);
		\fill (7,4) circle (\dotPt);
		\fill (9,4) circle (\dotPt);
		\fill (11,4) circle (\dotPt);

		\draw (1,0) -- (1,4);
		\draw (3,0) -- (3,4);
		\draw (9,0) -- (9,4);
		\draw (11,0) -- (11,4);
		\draw (5,0) -- (7,4);
		\draw[white, line width=8pt] (7-0.25,0+0.5) -- (5+0.25,4-0.5); 
		\draw (7,0) -- (5,4);
	
		\node at (6,-1) {$\sigma_i$};
		\node at (1,5) {$1$};
		\node at (11,5) {$n$};
		\node at (5,5) {$i$};
		\node at (7,5) {$i+1$};
		\node at (2, 2) {$\dots$};
		\node at (10, 2) {$\dots$};
	\end{scope}

	\begin{scope}[shift={(14,0)}]
		\draw (0,0) -- (12,0);
		\draw (0,4) -- (12,4);

		\def \dotPt {4pt}; 

		\fill (1,0) circle (\dotPt);
		\fill (3,0) circle (\dotPt);
		\fill (5,0) circle (\dotPt);
		\fill (7,0) circle (\dotPt);
		\fill (9,0) circle (\dotPt);
		\fill (11,0) circle (\dotPt);
	
		\fill (1,4) circle (\dotPt);
		\fill (3,4) circle (\dotPt);
		\fill (5,4) circle (\dotPt);
		\fill (7,4) circle (\dotPt);
		\fill (9,4) circle (\dotPt);
		\fill (11,4) circle (\dotPt);

		\draw (1,0) -- (1,4);
		\draw (3,0) -- (3,4);
		\draw (9,0) -- (9,4);
		\draw (11,0) -- (11,4);
		\draw (7,0) -- (5,4);
		\draw[white, line width=8pt] (5+0.25,0+0.5) -- (7-0.25,4-0.5); 
		\draw (5,0) -- (7,4);
	
		\node at (6,-1) {$\sigma_i^{-1}$};
		\node at (1,5) {$1$};
		\node at (11,5) {$n$};
		\node at (5,5) {$i$};
		\node at (7,5) {$i+1$};
		\node at (2, 2) {$\dots$};
		\node at (10, 2) {$\dots$};
	\end{scope}

	\end{tikzpicture}
	\caption{Braid diagrams for generators $\sigma_i, \sigma_i^{-1}$.}
	\label{}
\end{figure}
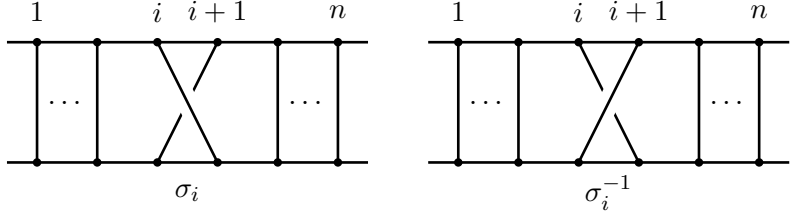

Multiplication in the braid group corresponds to concatenating diagrams from the top to the bottom, and the braid relations correspond to continuous deformations of the diagrams.
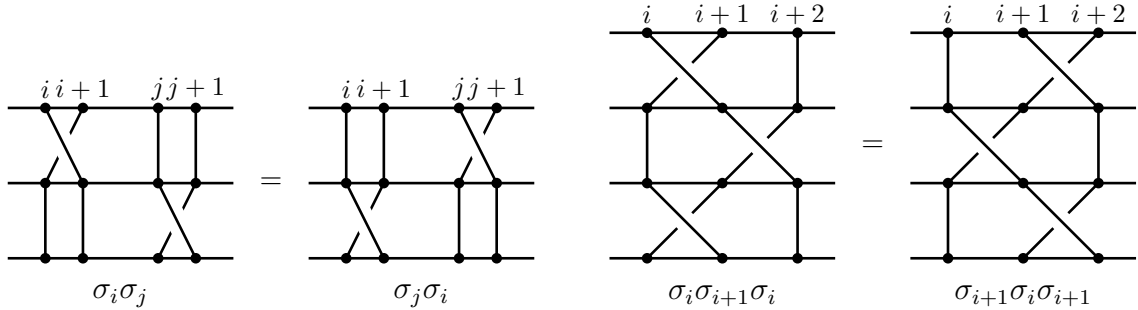
\begin{figure}[H]
	\centering
	\begin{tikzpicture}[line width=1pt, scale=0.5]
	\begin{scope}[shift={(0,0)}]
		\draw (0,0) -- (6,0);
		\draw (0,2) -- (6,2);
		\draw (0,4) -- (6,4);
	
		\def \dotPt {4pt}; 

		\fill (1,0) circle (\dotPt);
		\fill (2,0) circle (\dotPt);
		\fill (4,0) circle (\dotPt);
		\fill (5,0) circle (\dotPt);

		\fill (1,2) circle (\dotPt);
		\fill (2,2) circle (\dotPt);
		\fill (4,2) circle (\dotPt);
		\fill (5,2) circle (\dotPt);
	
		\fill (1,4) circle (\dotPt);
		\fill (2,4) circle (\dotPt);
		\fill (4,4) circle (\dotPt);
		\fill (5,4) circle (\dotPt);

		\draw (1,0) -- (1,2);
		\draw (2,0) -- (2,2);
		\draw (4,2) -- (4,4);
		\draw (5,2) -- (5,4);	
		\draw (4,0) -- (5,2);
		\draw[white, line width=8pt] (5-0.25,0+0.5) -- (4+0.25,2-0.5); 
		\draw (5,0) -- (4,2);
		\draw (1,2) -- (2,4);
		\draw[white, line width=8pt] (2-0.25,2+0.5) -- (1+0.25,4-0.5); 
		\draw (2,2) -- (1,4);
	
		\node at (3,-1) {\large $\sigma_i \sigma_j$};
		\node at (1,4.5) {\small $i$};
		\node at (2,4.5) {\small $i+1$};
		\node at (4,4.5) {\small $j$};
		\node at (5,4.5) {\small $j+1$};
		\node at (7,2) {$=$};
	\end{scope}

	\begin{scope}[shift={(8,0)}]
		\draw (0,0) -- (6,0);
		\draw (0,2) -- (6,2);
		\draw (0,4) -- (6,4);

		\def \dotPt {4pt}; 

		\fill (1,0) circle (\dotPt);
		\fill (2,0) circle (\dotPt);
		\fill (4,0) circle (\dotPt);
		\fill (5,0) circle (\dotPt);

		\fill (1,2) circle (\dotPt);
		\fill (2,2) circle (\dotPt);
		\fill (4,2) circle (\dotPt);
		\fill (5,2) circle (\dotPt);
	
		\fill (1,4) circle (\dotPt);
		\fill (2,4) circle (\dotPt);
		\fill (4,4) circle (\dotPt);
		\fill (5,4) circle (\dotPt);

		\draw (1,2) -- (1,4);
		\draw (2,2) -- (2,4);
		\draw (4,0) -- (4,2);
		\draw (5,0) -- (5,2);
		
		\draw (4,2) -- (5,4);
		\draw[white, line width=8pt] (5-0.25,2+0.5) -- (4+0.25,4-0.5); 
		\draw (5,2) -- (4,4);
		\draw (1,0) -- (2,2);
		\draw[white, line width=8pt] (2-0.25,0+0.5) -- (1+0.25,2-0.5); 
		\draw (2,0) -- (1,2);
	
		\node at (3,-1) {\large $\sigma_j \sigma_i$};
		\node at (1,4.5) {\small $i$};
		\node at (2,4.5) {\small $i+1$};
		\node at (4,4.5) {\small $j$};
		\node at (5,4.5) {\small $j+1$};
	\end{scope}

	\begin{scope}[shift={(16,0)}]
		\draw (0,0) -- (6,0);
		\draw (0,2) -- (6,2);
		\draw (0,4) -- (6,4);
		\draw (0,6) -- (6,6);
		
		\def \dotPt {4pt}; 
	
		\fill (1,0) circle (\dotPt);
		\fill (3,0) circle (\dotPt);
		\fill (5,0) circle (\dotPt);
	
		\fill (1,2) circle (\dotPt);
		\fill (3,2) circle (\dotPt);
		\fill (5,2) circle (\dotPt);
		
		\fill (1,4) circle (\dotPt);
		\fill (3,4) circle (\dotPt);
		\fill (5,4) circle (\dotPt);
		
		\fill (1,6) circle (\dotPt);
		\fill (3,6) circle (\dotPt);
		\fill (5,6) circle (\dotPt);
	
		\draw (1,2) -- (1,4);
		\draw (5,0) -- (5,2);
		\draw (5,4) -- (5,6);
		
		\draw (1,0) -- (3,2);
		\draw[white, line width=8pt] (3-0.5,0+0.5) -- (1+0.5,2-0.5); 
		\draw (3,0) -- (1,2);
	
		\draw (1,4) -- (3,6);
		\draw[white, line width=8pt] (3-0.5,4+0.5) -- (1+0.5,6-0.5); 
		\draw (3,4) -- (1,6);
		
		\draw (3,2) -- (5,4);
		\draw[white, line width=8pt] (5-0.5,2+0.5) -- (3+0.5,4-0.5); 
		\draw (5,2) -- (3,4);
		
		\node at (3,-1) {\large $\sigma_i \sigma_{i+1} \sigma_i$};
		\node at (1,6.5) {\small $i$};
		\node at (3,6.5) {\small $i+1$};
		\node at (5,6.5) {\small $i+2$};
		\node at (7,3) {$=$};
	\end{scope}

	\begin{scope}[shift={(24,0)}]
		\draw (0,0) -- (6,0);
		\draw (0,2) -- (6,2);
		\draw (0,4) -- (6,4);
		\draw (0,6) -- (6,6);
	
		\def \dotPt {4pt}; 

		\fill (1,0) circle (\dotPt);
		\fill (3,0) circle (\dotPt);
		\fill (5,0) circle (\dotPt);
	
		\fill (1,2) circle (\dotPt);
		\fill (3,2) circle (\dotPt);
		\fill (5,2) circle (\dotPt);
		
		\fill (1,4) circle (\dotPt);
		\fill (3,4) circle (\dotPt);
		\fill (5,4) circle (\dotPt);
		
		\fill (1,6) circle (\dotPt);
		\fill (3,6) circle (\dotPt);
		\fill (5,6) circle (\dotPt);
	
		\draw (1,0) -- (1,2);
		\draw (1,4) -- (1,6);
		\draw (5,2) -- (5,4);
		
		\draw (1,2) -- (3,4);
		\draw[white, line width=8pt] (3-0.5,2+0.5) -- (1+0.5,4-0.5); 
		\draw (3,2) -- (1,4);
	
		\draw (3,4) -- (5,6);
		\draw[white, line width=8pt] (5-0.5,4+0.5) -- (3+0.5,6-0.5); 
		\draw (5,4) -- (3,6);
	
		\draw (3,0) -- (5,2);
		\draw[white, line width=8pt] (5-0.5,0+0.5) -- (3+0.5,2-0.5); 
		\draw (5,0) -- (3,2);
		
		\node at (3,-1) {\large $\sigma_{i+1} \sigma_i \sigma_{i+1}$};
		\node at (1,6.5) {\small $i$};
		\node at (3,6.5) {\small $i+1$};
		\node at (5,6.5) {\small $i+2$};
	\end{scope}

	\end{tikzpicture}
	\caption{Braid diagrams for the braid relations.}
	\label{}
\end{figure}

For each braid $\sigma \in B_n$, its \emph{closure} $\widehat{\sigma}$ is obtained by connecting the $i$-th top endpoint to the $i$-th bottom endpoint for all $1 \le i \le n$, as illustrated in the figure. This generally produces a link, which consists of one or more knots.
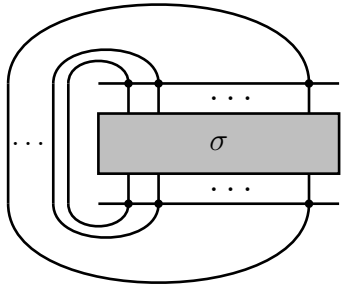
\begin{figure}[H]
\centering
\begin{tikzpicture}[line width=1pt, scale=0.4]
\begin{scope}[shift={(0,0)}]
	\draw  (0,0) -- (8,0);
	\draw  (0,4) -- (8,4);

	\def \dotPt {4pt}; 

	\fill (1,0) circle (\dotPt);
	\fill (2,0) circle (\dotPt);
	\fill (7,0) circle (\dotPt);

	\fill (1,4) circle (\dotPt);
	\fill (2,4) circle (\dotPt);
	\fill (7,4) circle (\dotPt);

	\draw (1,0) -- (1,4);
	\draw (2,0) -- (2,4);
	\draw (7,0) -- (7,4);
	
	\draw (1,0) .. controls +(0,-1) and +(0,-1) .. (-1,0);
	\draw (2,0) .. controls +(0,-1.5) and +(0,-1.5) .. (-1.5,0);
	\draw (7,0) ..  controls +(0,-3.5) and +(0,-3.5) .. (-3,0);
	
	\draw (1,4) .. controls +(0, 1) and +(0, 1) .. (-1,4);
	\draw (2,4) .. controls +(0,1.5) and +(0,1.5) .. (-1.5,4);
	\draw (7,4) ..  controls +(0, 3.5) and +(0,3.5) .. (-3,4);
	
	\draw (-1, 0) -- (-1,4);
	\draw (-1.5, 0) -- (-1.5, 4);
	\draw (-3, 0) -- (-3, 4);

	\draw[fill=lightgray, draw=black] (0,1) rectangle (8,3);
	
	\node at (4, 2) {\large $\sigma$};
	\node at (4.5, 3.5) {\Large $\dots$};
	\node at (4.5, 0.5) {\Large $\dots$};
	\node at (-2.25, 2) {$\dots$};
\end{scope}
\end{tikzpicture}
\caption{The closure of a braid.}
\label{}
\end{figure}

\begin{definition}
	A braid $\sigma \in B_n$ is called a \emph{positive braid} if it can be expressed as a product of positive powers of the generators $\sigma_j$. 
\end{definition}

\subsection{Continued fraction expansions}

For a positive (non-square) discriminant $D$ and a quadratic form $Q(x,y) = ax^2+ bxy + cy^2 \in \mathcal{Q}_D$, we associate one of its roots, a real quadratic irrational,
\[
	w_Q = \frac{-b + \sqrt{D}}{2a}.
\]
For a real number $x$, we consider the (regular) \emph{continued fraction expansion}
\[
	x = a_1 + \cfrac{1}{a_2 + \cfrac{1}{a_3 + \cfrac{1}{\ddots}}} = [a_1, a_2, a_3, \dots],
\]
where $a_1 \in \Z$ and $a_j \in \Z_{>0}$ for $j \ge 2$. It is known that the continued fraction expansion of a real number is eventually periodic if and only if the number is a real quadratic irrational, (see~\cite[Theorem 1.17]{Aigner2013}). Hence, we can write
\begin{align}\label{eq:cont-frac-wQ}
	w_Q = [a_1, \dots, a_{2k}, \overline{c_1, c_2, \dots, c_{2\ell}}]
\end{align}
with even $2k$ and minimal even period $2\ell$. Here, the term ``minimal even period" means the following. Let $(c_1, \dots, c_m)$ be the minimal period of the continued fraction of $w_Q$. If $m$ is even, we set $2\ell = m$. If $m$ is odd, we use the doubled block $(c_1, \dots, c_m, c_1, \dots, c_m)$ as its period and set $2\ell = 2m$.

\begin{example}
	For a quadratic form $Q(x,y) = x^2 - 3xy + y^2$, we have
	\[
		w_Q = \frac{3 + \sqrt{5}}{2} = [2, 1, \overline{1,1}].
	\]
	Note that, since we assume both $2k$ and $2\ell$ are even, we do not adopt expansions such as $[2, \overline{1}]$. Similarly, for $Q(x,y) = 5x^2-6xy-3y^2$, we have
	\begin{align}\label{eq:ex-wQ-cont}
		w_Q = \frac{3 + 2\sqrt{6}}{5} = [\overline{1,1,1,2}] = [1,1, \overline{1,2,1,1}].
	\end{align}
	Here, since no minimality condition is imposed on the non-periodic part $(a_1, \dots, a_{2k})$, some arbitrariness remains, as seen in this example.
\end{example}

Let $\{L, R\}^*_\mathrm{prim}$ denote the set of all finite-length primitive words over the two-letter alphabet $\{L, R\}$, where a word is called \emph{primitive} if it cannot be written as a repetition of a shorter word. 
We define an equivalence relation on this set by cyclic permutations and denote the corresponding quotient set by $\{L, R\}^*_\mathrm{prim}/{\sim}$. This gives rise to the following map.

\begin{lemma}\label{lem:cont-word-corr}
	The map $W: \calQ_+/\SL_2(\Z) \to \{L, R\}^*_\mathrm{prim}/{\sim}$ defined by
	\[
		[Q] \mapsto W(Q) = L^{c_1} R^{c_2} \cdots L^{c_{2\ell-1}} R^{c_{2\ell}}
	\]
	is well-defined, where $c_1, \dots, c_{2\ell}$ are defined as in~\eqref{eq:cont-frac-wQ}. Moreover, this map is surjective onto all elements except for the classes of primitive words of length one, $L$ and $R$.
\end{lemma}

\begin{proof}
	The well-definedness follows from the relation $w_{Q \circ \gamma} = \gamma^{-1} w_Q$ (which can be verified by direct calculation), together with the fact that two real quadratic irrationals are $\SL_2(\Z)$-equivalent if and only if, when expanded in the form of \eqref{eq:cont-frac-wQ}, they share the same periodic part $(c_1, \dots, c_{2\ell})$ in their continued fraction expansions (a result due to Hurwitz~\cite[p.434]{Hurwitz1894}, see also~\cite[Appendix]{Reutenauer2019}). Note that in \eqref{eq:cont-frac-wQ}, the non-periodic part $(a_1, \dots, a_{2k})$ can be chosen arbitrarily, which accounts for the ambiguity corresponding to an even cyclic shift of the periodic part. 
	
	As for the surjectivity excluding $\{L, R\}$, note first that any primitive word of length at least $2$ can be written in the form $L^{c_1} R^{c_2} \cdots L^{c_{2\ell-1}} R^{c_{2\ell}}$ after a suitable cyclic permutation. Then, for matrices $T = \smat{1 & 1 \\ 0 & 1}, V = \smat{1 & 0 \\ 1 & 1} \in \SL_2(\Z)$, we consider the matrix $\gamma = \smat{a & b \\ c & d} = T^{c_1} V^{c_2} \cdots T^{c_{2\ell-1}} V^{c_{2\ell}}$ and define a quadratic form
	\[
		Q(x,y) = cx^2 + (d-a)xy - by^2 \in \calQ_+.
	\]
	Since $c > 0$, its root $w_Q$ is the larger fixed point of $\gamma$, that is, $w_Q = [\overline{c_1, c_2, \dots, c_{2\ell}}]$. Therefore, for this $Q$, the word $W(Q)$ coincides with the given word.
\end{proof}

The Lyndon words defined below are commonly used as a canonical system of representatives for $\{L, R\}^*_\mathrm{prim}/{\sim}$.

\begin{definition}
	A word $W \in \{L,R\}_\mathrm{prim}^*$ is called a \emph{Lyndon word} if every nontrivial cyclic permutation of $W$ is strictly greater than $W$ in the lexicographic order with $L < R$.
\end{definition}

\begin{example}\label{ex:Q365}
	For $Q(x,y) = 5x^2 - 6xy -3y^2$, the continued fraction expansion in \eqref{eq:ex-wQ-cont} and the map in \cref{lem:cont-word-corr} yield the primitive word $L R L R^2$ or $L R^2 L R$. Although there is some arbitrariness depending on which the continued fraction expansion is chosen, we ignore differences up to cyclic permutations, so these words define the same class in $\{L, R\}^*_\mathrm{prim}/{\sim}$. Among the words in this class, the minimal element in lexicographic order is $L R L R^2$. This is the Lyndon word and serves as the canonical representative of the class.
\end{example}

\subsection{Modular knots using the Bernoulli shift}

We give two equivalent definitions of modular knots. The first uses the Bernoulli shift, and the second simply reformulates the first definition in terms of words. Our exposition here is based on the survey by Dehornoy~\cite{Dehornoy2011}. First, the Bernoulli shift is the following discrete dynamical system, which arises as a simple example of chaotic dynamics.

\begin{definition}
	The \emph{Bernoulli shift} is a discrete dynamical system on $[0,1]$, defined by the map $\mathrm{Ber}: [0,1] \to [0,1]$ with $\mathrm{Ber}(x) = 2x \pmod{1}$.
\end{definition}

For any rational number in $[0,1]$, the Bernoulli shift produces a periodic orbit. For example, taking $x = 11/31$ as the initial value, it gives the orbit
\[
	\frac{11}{31} \to \frac{22}{31} \to \frac{13}{31} \to \frac{26}{31} \to \frac{21}{31} \to \frac{11}{31} \to \cdots
\]
with period $5$. 

We now explain how to construct a braid diagram from such a periodic orbit. Place the interval $[0,1]$ at both the top and the bottom of the diagram. For each point $x \in [0,1]$ in the periodic orbit, connect the point $x$ at the top to $\mathrm{Ber}(x)$ at the bottom with a line segment. Repeating this for all points produces a braid diagram with as many strands as the period of the orbit. By the definition of the Bernoulli shift, segments starting from $[0,1/2]$ do not cross each other, and similarly for $[1/2, 1]$. Whenever the two segments cross, the segment from $[0,1/2]$ passes over the one from $[1/2, 1]$, yielding a positive braid crossing.
The following figure illustrates the case of $x = 11/31$.
\begin{figure}[H]
	\centering
	\begin{tikzpicture}[line width=1pt, scale=0.5]
	\begin{scope}[shift={(0,0)}]
		\draw (0,0) -- (10,0);
		\draw (0,4) -- (10,4);

		\def \dotPt {4pt}; 

		\fill (110/31,4) circle (\dotPt);
		\fill (130/31,4) circle (\dotPt);
		\fill (210/31,4) circle (\dotPt);
		\fill (220/31,4) circle (\dotPt);
		\fill (260/31,4) circle (\dotPt);

		\draw (210/31,4) -- (110/31,0);
		\draw (220/31,4) -- (130/31,0);
		\draw (260/31,4) -- (210/31,0);
	
		\draw[white, line width=4pt] (220/31-1/8, 31/220) -- (110/31+1/2,189/55); 
		\draw (220/31,0) -- (110/31,4);
		\draw[white, line width=4pt] (260/31-0.5, 31/65) -- (130/31+0.5,229/65); 
		\draw (260/31,0) -- (130/31,4);

		\fill (110/31,0) circle (\dotPt);
		\fill (130/31,0) circle (\dotPt);
		\fill (210/31,0) circle (\dotPt);
		\fill (220/31,0) circle (\dotPt);
		\fill (260/31,0) circle (\dotPt);
	
		\node at (110/31-0.2,5) {\small $\frac{11}{31}$};
		\node at (130/31,5) {\small $\frac{13}{31}$};
		\node at (210/31-0.5,5) {\small $\frac{21}{31}$};
		\node at (220/31,5) {\small $\frac{22}{31}$};
		\node at (260/31,5) {\small $\frac{26}{31}$};
	
		\node at (110/31-0.2,-1) {\small $\frac{11}{31}$};
		\node at (130/31,-1) {\small $\frac{13}{31}$};
		\node at (210/31-0.5,-1) {\small $\frac{21}{31}$};
		\node at (220/31,-1) {\small $\frac{22}{31}$};
		\node at (260/31,-1) {\small $\frac{26}{31}$};
	
		\draw (0,0) node[left]{$0$};
		\draw (0,4) node[left]{$0$};
		\draw (10,0) node[right]{$1$};
		\draw (10,4) node[right]{$1$};
	\end{scope}

	\end{tikzpicture}
	\caption{A braid diagram defined from the periodic orbit of the Bernoulli shift for $11/31$.}
	\label{Bernoulli-shift-trefoil}
\end{figure}
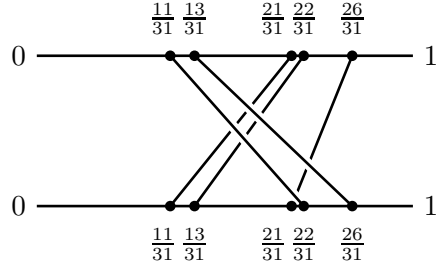

For a quadratic form class $[Q]$, or equivalently for the class of the associated real quadratic irrational $[w_Q]$, we define the corresponding modular knot as follows.

\begin{definition}\label{def:modular-knot-Ber}
	For $[Q] \in \calQ_+/\SL_2(\Z)$, replacing $L$ with $0$ and $R$ with $1$ in the word $W(Q)$ from \cref{lem:cont-word-corr} gives a rational number $x \in [0,1]$ with binary expansion
	\[
		x = 0. \overline{0^{c_1} 1^{c_2} \cdots 0^{c_{2\ell-1}} 1^{c_{2\ell}}}_{(2)}.
	\]
	Here, for $j \in \{0,1\}$, $j^c$ denotes $c$ consecutive copies of $j$. The braid obtained from the Bernoulli shift orbit of $x$ is called a \emph{modular braid} associated with $[Q]$, denoted by $\mathbb{M}(Q)$ or $\mathbb{M}(w_Q)$. Its closure defines a knot, which is called a \emph{modular knot}.
\end{definition}

Note that $x$ depends on the choice of a cyclic permutation of $W(Q)$, whereas $\mathbb{M}(Q)$ does not.

\begin{example}
	For $Q(x,y) = 5x^2 - 6xy - 3y^2$, as seen in~\cref{ex:Q365}, we have $W(Q) = LRLR^2$. This gives the rational number
	\[
		0.\overline{01011}_{(2)} = \sum_{j=0}^\infty \left(\frac{1}{2^{5j+2}} + \frac{1}{2^{5j+4}} + \frac{1}{2^{5j+5}} \right) = \frac{11}{31},
	\]
	so the braid obtained in \cref{Bernoulli-shift-trefoil} is the associated modular braid. Its closure is the modular knot, which is isotopic to the trefoil.
\end{example}

\subsection{Modular knots using words}

The discussion in the previous subsection can also be reformulated purely in terms of words, without referring to the Bernoulli shift. The number of letters contained in a word $W$ is called its \emph{length}, and we write it as $\mathrm{length}(W)$. We denote by $\mathrm{cyc}: \{L, R\}_\mathrm{prim}^* \to \{L, R\}_\mathrm{prim}^*$ the cyclic permutation that moves the first letter of a word to the end. 

\begin{definition}\label{def:Lyndon-order}
	For a word $W \in \{L, R\}^*_\mathrm{prim}$, let $\mathrm{rank}_j(W)$ $(1 \le j \le \mathrm{length}(W)$) denote the ascending-order rank of the cyclic permutation $\mathrm{cyc}^{\circ(j-1)}(W)$ among all cyclic permutations of $W$. We also write $m_W$ for the value of $j$ such that $\mathrm{rank}_j(W) = \mathrm{length}(W)$. In other words, the permutation obtained after $(m_W-1)$ moves is the largest among all cyclic permutations of $W$.
\end{definition}

\begin{example}
	For the word $W = LRLRR$, its cyclic permutations are given by $LRLRR$, $RLRRL$, $LRRLR$, $RRLRL$, and $RLRLR$. Arranging them in lexicographic order gives
	\[
		LRLRR < LRRLR < RLRLR < RLRRL < RRLRL.
	\]
	Since $W = LRLRR$ is the smallest among these, it is a Lyndon word. Being minimal among its cyclic permutations, we have $\mathrm{rank}_1(W) = 1$. After one cyclic shift, we obtain $RLRRL$, which is the fourth smallest among all cyclic permutations, giving $\mathrm{rank}_2(W) = 4$. The maximal permutation is $RRLRL$, which results from $4-1 = 3$ shifts, so $m_W = 4$.
\end{example}

Without referring to the Bernoulli shift, the description can be given as follows. For each word $W \in \{L,R\}^*_\mathrm{prim}$, we define a sequence $(\mathrm{rank}_j(W))_j$ by considering all of its cyclic permutations. Then we construct a positive braid diagram by connecting the $\mathrm{rank}_j(W)$-th point on the top to the $\mathrm{rank}_{j+1}(W)$-th point on the bottom for all $1 \le j \le \mathrm{length}(W)$, where we set $\mathrm{rank}_{\mathrm{length}(W)+1}(W) = \mathrm{rank}_1(W)$.

\begin{example}\label{eg:trefoil}
	For the Lyndon word $W = LRLRR$, its cyclic permutations are given as follows:
	\[
	\begin{array}{c|ccccc}
		\mathrm{cyc}^{\circ(j-1)}(W) & LRLRR & RLRRL & LRRLR & RRLRL & RLRLR \\
		\hline
		\mathrm{rank}_j(W) & 1 & 4 & 2 & 5 & 3
	\end{array}
	\]
	The corresponding braid is given below, and coincides with the one shown in \cref{Bernoulli-shift-trefoil}.
	\begin{figure}[H]
	\centering
	\begin{tikzpicture}[line width=1pt, scale=0.5]
	\begin{scope}[shift={(0,0)}]

		\def\dotPt{4pt}

		\draw (0,0) -- (10,0);
		\draw (0,4) -- (10,4);

	\foreach \x in {1,3,5,7,9}{
	  \fill (\x,0) circle (\dotPt);
	  \fill (\x,4) circle (\dotPt);
	}
	
	\foreach \a/\b in {1/5, 3/7, 5/9}{
	  \draw (\a,0) -- (\b,4);
	}
	
	\foreach \a/\b in {7/1, 9/3}{
	  \draw[white, line width=8pt] (\a-0.5,0+1/3) -- (\b+0.5,11/3);
	  \draw (\a,0) -- (\b,4);
	}
	
	\foreach \i/\x in {1/1, 2/3, 3/5, 4/7, 5/9}{
	  \node at (\x,5) {$\i$};
	  \node at (\x,-1) {$\i$};
	}
	\end{scope}
	\end{tikzpicture}
		\caption{A braid diagram for the Lyndon word $LRLRR$.}
		\label{Lorenz-Lyndon-ababb}
	\end{figure}
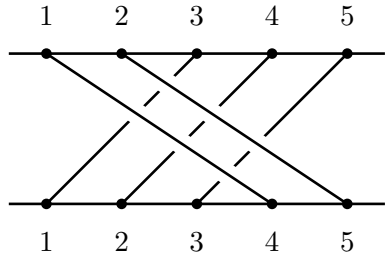
\end{example}

\begin{definition}\label{def:modular-knot-word}
	For a word $W \in \{L, R\}_{\mathrm{prim}}^*$, the braid obtained above is called a \emph{modular braid} associated with $W$, denoted by $\mathbb{M}(W)$. Its closure defines a knot, which is called a \emph{modular knot}.
\end{definition}

We have now defined the modular knot in two different ways, in \cref{def:modular-knot-Ber} and \cref{def:modular-knot-word}. These two definitions are equivalent in the following sense.

\begin{proposition}\label{prop:two-modular-braids}
	For $[Q] \in \calQ_+/\SL_2(\Z)$, we have $\mathbb{M}(W(Q)) = \mathbb{M}(Q)$. Conversely, for any primitive word $W \in \{L, R\}_\mathrm{prim}^* \setminus \{L, R\}$, there exists a quadratic form $Q \in \mathcal{Q}_+$ such that $\mathbb{M}(W) = \mathbb{M}(Q)$.
\end{proposition}

\begin{proof}
	For a rational number $x \in [0,1]$, let $n$ be the period of the periodic orbit of $x$ under the Bernoulli shift. We define $\mathrm{rank}_j(x)$ ($1 \le j \le n$) as the ascending-order rank of $\mathrm{Ber}^{\circ(j-1)}(x)$ within its periodic orbit. Since the injective map $\phi: \{L, R\}_\mathrm{prim}^* \setminus \{L, R\} \to (0,1) \cap \Q$ given in \cref{def:modular-knot-Ber} by
	\[
		\phi(W) = 0. \overline{W|_{L \to 0, R \to 1}}_{(2)}
	\]
	is order preserving, and satisfies $\phi \circ \mathrm{cyc} = \mathrm{Ber} \circ \phi$, we have $\mathrm{rank}_j(W) = \mathrm{rank}_j(\phi(W))$. Therefore, a quadratic form $Q$ and a word $W(Q)$ define the same braid. The latter follows from \cref{lem:cont-word-corr}. Indeed, for any $W$, there exists $Q \in \calQ_+$ such that a cyclic permutation of $W(Q)$ coincides with $W$. Hence, $\mathbb{M}(W) = \mathbb{M}(W(Q)) = \mathbb{M}(Q)$.
\end{proof}


\section{Lorenz braids and Lorenz knots}\label{sec:Lorenz-braid}

\subsection{Lorenz braids}

More generally, we also introduce the notion of the Lorenz braid. For the history of research on Lorenz braids, see the article by Birman--Williams~\cite{BirmanWilliams1983}, as well as surveys by Birman~\cite{Birman2013} and Dehornoy~\cite{Dehornoy2011}. 

\begin{definition}\label{def:Lorenz-braid}
	Given a positive integer $r$ and pairs of positive integers $(p_1, q_1), \dots, (p_r, q_r)$, define $P_j = p_1 + \cdots + p_j$ and $Q_j = q_1 + \cdots + q_j$ ($1 \le j \le r$). The \emph{Lorenz braid} associated with these data is constructed as follows:
	\begin{enumerate}
		\item The braid has $P_r + Q_r$ strands.
		\item The top row is divided into two parts, the left part of size $Q_r$ and the right part of size $P_r$. The bottom row is divided, from left to right, into $2r$ parts of sizes $p_1, q_1, \dots, p_r, q_r$. Each part of size $p_j$ (resp.~$q_j$) is called a $p$-part (resp.~$q$-part).
		\item Strands from the left top part are connected, in order, to the $q$-parts, and strands from the right top part are connected to the $p$-parts.
		\item Crossings are arranged so that strands from the left top part pass over those from the right top part, making the braid positive.
	\end{enumerate}
	We denote this braid by $\mathbb{L}(p_1, q_1; \dots; p_r, q_r)$,	and its closure is called a \emph{Lorenz link}. When the closure is a knot, we refer to it as a \emph{Lorenz knot}.
\end{definition}

\begin{example}
	For $r=4$ and $(p_1, q_1; p_2, q_2; p_3, q_3; p_4, q_4) = (2,4; 1,2; 3,1; 2,2)$, we have the following Lorenz braid. The same example was considered under different parameters in~\cite[Figure 3]{BirmanKofman2009}.
\begin{figure}[H]
	\centering
	\begin{tikzpicture}[line width=1pt, scale=0.45]
	\begin{scope}[shift={(0,0)}]
		\draw (0,0) -- (18,0);
		\draw (0,6) -- (18,6);
		
		\def \dotPt {4pt}; 
		\foreach \x in {1,...,17}{\fill (\x,0) circle (\dotPt);}
		\foreach \x in {1,...,17}{\fill (\x,6) circle (\dotPt);}
	
		\draw (1,0) -- (10,6);
		\draw (2,0) -- (11,6);
		\draw (7,0) -- (12,6);
		\draw (10,0) -- (13,6);
		\draw (11,0) -- (14,6);
		\draw (12,0) -- (15,6);
		\draw (14,0) -- (16,6);
		\draw (15,0) -- (17,6);
		\draw[white, line width=4pt] (3-1/8,0+3/8) -- (1+1/8,45/8); 
		\draw (3,0) -- (1,6);
		\draw[white, line width=4pt] (4-1/8,0+3/8) -- (2+1/8,45/8); 
		\draw (4,0) -- (2,6);
		\draw[white, line width=4pt] (5-1/8,0+3/8) -- (3+1/8,45/8); 
		\draw (5,0) -- (3,6);
		\draw[white, line width=4pt] (6-1/8,0+3/8) -- (4+1/8,45/8); 
		\draw (6,0) -- (4,6);
		\draw[white, line width=4pt] (8-1/8,0+1/4) -- (5+1/8,23/4); 
		\draw (8,0) -- (5,6);
		\draw[white, line width=4pt] (9-1/8,0+1/4) -- (6+1/8,23/4); 
		\draw (9,0) -- (6,6);
		\draw[white, line width=5pt] (13-1/4,0+1/4) -- (7+1/4,23/4); 
		\draw[line width=3pt] (13,0) -- (7,6);
		\draw[white, line width=5pt] (16-1/4,0+3/16) -- (8+1/4,93/16); 
		\draw[line width=3pt] (16,0) -- (8,6);
		\draw[white, line width=5pt] (17-1/4,0+3/16) -- (9+1/4,93/16); 
		\draw[line width=3pt] (17,0) -- (9,6);
		
		\foreach \x in {1,...,17}{\node at (\x,6.5) {\tiny $\x$};}
		\foreach \x in {1,...,17}{\node at (\x,-0.5) {\tiny $\x$};}
		
		\draw (1,8) -- (9.3,8); 
		\draw (1,8) -- (1,7.5);
		\draw (9.3,8) -- (9.3,7.5);
		\draw (17,8) -- (9.7,8); 
		\draw (17,8) -- (17,7.5);
		\draw (9.7,8) -- (9.7,7.5);
		\node at (5.5, 7.5) {$L$}; \node at (13.5, 7.5) {$R$};
	\end{scope}
	
	\end{tikzpicture}
	\caption{Lorenz braid for $\mathbb{L}(2,4; 1,2; 3,1; 2,2)$.}
	\label{Lorenz-braid-example}
\end{figure}
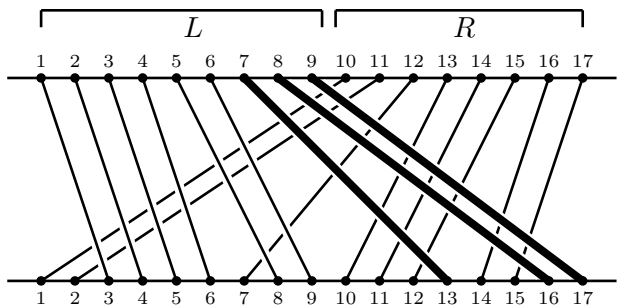
\end{example}

By definition, a modular braid is a special case of a Lorenz braid. In particular, we have the map
\[
	\mathbb{M}: (\{L, R\}_\mathrm{prim}^* \setminus \{L, R\})/{\sim} \to \{\text{Lorenz braids}\}.
\]
However, the closure of a Lorenz braid is generally a link with multiple components rather than a knot. Hence, the above map is not surjective, that is, Lorenz braids form a broader family than modular braids. 

On the other hand, Birman--Williams~\cite[Proposition 3.1]{BirmanWilliams1983} showed that any Lorenz knot corresponds one-to-one with a class of primitive $L$-$R$ words. The associated $L$-$R$ word is obtained, as explained in the following example, by following the orbit of points in the Lorenz braid and assigning the letter $L$ (resp.~$R$) to each point lying in the left (resp.~right) part at the top. Ghys~\cite{Ghys2006} related these word expressions to the expressions of hyperbolic elements of $\SL_2(\Z)$ using its two generators, thereby establishing a connection between modular knots and Lorenz knots. In this article, we instead described modular braids using the continued fraction expansions of real quadratic irrationals, rather than the generators of $\SL_2(\Z)$.

\begin{example}
	The Lorenz braid $\mathbb{L}(2,4; 1,2; 3,1; 2,2)$ given in \cref{Lorenz-braid-example} is also the modular braid $\mathbb{M}(W)$ associated with the word 
	\begin{align}\label{eq:Lorenz-LR-word}
		W = L^4 R^3 L^4 R^3 L R^2.
	\end{align}
	In this Lorenz braid, the transition of points is
	\[
		1 \to 3 \to 5 \to 8 \to 16 \to 14 \to 11 \to 2 \to 4 \to 6 \to 9 \to 17 \to 15 \to 12 \to 7 \to 13 \to 10 \to 1.
	\]
	According to the construction of modular braid, assigning $L$ to the integers $1, \dots, 9$ and $R$ to $10, \dots, 17$, and reading them in the order above, yields the word $W = L^4 R^3 L^4 R^3 L R^2$. The quadratic form corresponding to this word is given by $Q(x,y) = 152x^2 - 600xy - 237y^2$, whose root has the continued fraction expansion
	\begin{align}\label{eq:Lorenz-cont-frac}
		w_Q = \frac{150 + \sqrt{31506}}{76} = [\overline{4,3,4,3,1,2}].
	\end{align}
\end{example}

\subsection{Some remarks on Lorenz knots}

In the above definition, we consider arbitrary positive integers $p_1, q_1, \dots, p_r, q_r$. However, when focusing on its closure, we can assume $p_1 > 1$ and $q_r > 1$ without loss of generality by applying a Markov move, (see~\cite[Section 2.5]{KasselTuraev2008}). This result can also be restated in terms of the word description of modular braids.

\begin{lemma}\label{lem:Markov}
	Let $W$ be a Lyndon word. Define $W_L$ by adding the letter $L$ at the beginning of $W$, and $W_R$ by inserting the letter $R$ immediately after the $m_W$-th letter of $W$. Then the closures of $\mathbb{M}(W), \mathbb{M}(W_L)$, and $\mathbb{M}(W_R)$ are isotopic. Here, $m_W$ is defined as in~\cref{def:Lyndon-order}.
\end{lemma}

\begin{proof}
	These two operations correspond to Markov moves in the braid. See~\cite[Proposition 1.12]{Dehornoy2011} for details.
\end{proof}

\begin{example}
	For the braid in \cref{Lorenz-braid-example}, the corresponding Lyndon word is $L^4 R^3 L^4 R^3 LR^2$ as before. Among its cyclic permutations, the largest is $R^3 LR^2 L^4R^3L^4$, which is obtained after $11$ operations, so $m_W = 12$. Therefore, we have
	\[
		W_L = L^5R^3L^4R^3LR^2, \qquad W_R = L^4R^3L^4R^4LR^2.
	\]
	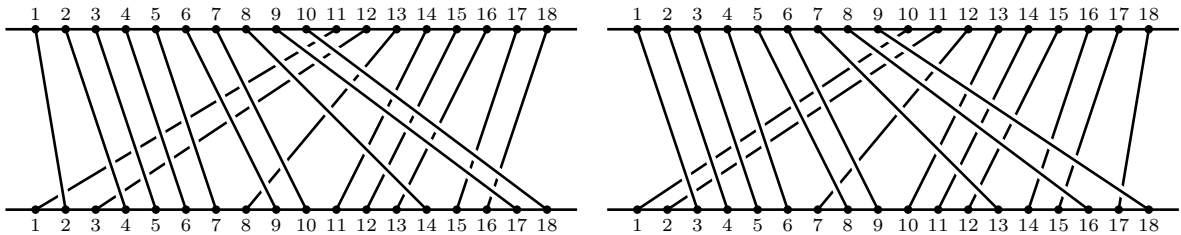
\begin{figure}[H]
	\centering
	\begin{tikzpicture}[line width=1pt, scale=0.4]
	\begin{scope}[shift={(0,0)}]
		\draw (0,0) -- (19,0);
		\draw (0,6) -- (19,6);
		
		\def \dotPt {4pt}; 
		\foreach \x in {1,...,18}{\fill (\x,0) circle (\dotPt);}
		\foreach \x in {1,...,18}{\fill (\x,6) circle (\dotPt);}
	
		\draw (1,0) -- (11,6);
		\draw (3,0) -- (12,6);
		\draw (8,0) -- (13,6);
		\draw (11,0) -- (14,6);
		\draw (12,0) -- (15,6);
		\draw (13,0) -- (16,6);
		\draw (15,0) -- (17,6);
		\draw (16,0) -- (18,6);
		
		\draw[white, line width=4pt] (2-1/16,3/8) -- (1+1/16, 45/8);
		\draw (2,0) -- (1,6);
		\draw[white, line width=4pt] (4-1/8,0+3/8) -- (2+1/8,45/8); 
		\draw (4,0) -- (2,6);
		\draw[white, line width=4pt] (5-1/8,0+3/8) -- (3+1/8,45/8); 
		\draw (5,0) -- (3,6);
		\draw[white, line width=4pt] (6-1/8,0+3/8) -- (4+1/8,45/8); 
		\draw (6,0) -- (4,6);
		\draw[white, line width=4pt] (7-1/8,0+3/8) -- (5+1/8,45/8); 
		\draw (7,0) -- (5,6);
		\draw[white, line width=4pt] (9-1/8,0+1/4) -- (6+1/8,23/4); 
		\draw (9,0) -- (6,6);
		\draw[white, line width=4pt] (10-1/8,0+1/4) -- (7+1/8,23/4); 
		\draw (10,0) -- (7,6);
		\draw[white, line width=4pt] (14-1/4,0+1/4) -- (8+1/4,23/4); 
		\draw (14,0) -- (8,6);
		\draw[white, line width=4pt] (17-1/4,0+3/16) -- (9+1/4,93/16); 
		\draw (17,0) -- (9,6);
		\draw[white, line width=4pt] (18-1/4,0+3/16) -- (10+1/4,93/16); 
		\draw (18,0) -- (10,6);
		
		\foreach \x in {1,...,18}{\node at (\x,6.5) {\tiny $\x$};}
		\foreach \x in {1,...,18}{\node at (\x,-0.5) {\tiny $\x$};}
	\end{scope}
	
	\begin{scope}[shift={(20,0)}]
		\draw (0,0) -- (19,0);
		\draw (0,6) -- (19,6);
		
		\def \dotPt {4pt}; 
		\foreach \x in {1,...,18}{\fill (\x,0) circle (\dotPt);}
		\foreach \x in {1,...,18}{\fill (\x,6) circle (\dotPt);}
	
		\draw (1,0) -- (10,6);
		\draw (2,0) -- (11,6);
		\draw (7,0) -- (12,6);
		\draw (10,0) -- (13,6);
		\draw (11,0) -- (14,6);
		\draw (12,0) -- (15,6);
		\draw (14,0) -- (16,6);
		\draw (15,0) -- (17,6);
		\draw (17,0) -- (18,6);
		
		\draw[white, line width=4pt] (3-1/8,0+3/8) -- (1+1/8,45/8); 
		\draw (3,0) -- (1,6);
		\draw[white, line width=4pt] (4-1/8,0+3/8) -- (2+1/8,45/8); 
		\draw (4,0) -- (2,6);
		\draw[white, line width=4pt] (5-1/8,0+3/8) -- (3+1/8,45/8); 
		\draw (5,0) -- (3,6);
		\draw[white, line width=4pt] (6-1/8,0+3/8) -- (4+1/8,45/8); 
		\draw (6,0) -- (4,6);
		\draw[white, line width=4pt] (8-1/8,0+1/4) -- (5+1/8,23/4); 
		\draw (8,0) -- (5,6);
		\draw[white, line width=4pt] (9-1/8,0+1/4) -- (6+1/8,23/4); 
		\draw (9,0) -- (6,6);
		\draw[white, line width=4pt] (13-1/4,0+1/4) -- (7+1/4,23/4); 
		\draw (13,0) -- (7,6);
		\draw[white, line width=4pt] (16-1/4,0+3/16) -- (8+1/4,93/16); 
		\draw (16,0) -- (8,6);
		\draw[white, line width=4pt] (18-1/3,0+2/9) -- (9+1/3,52/9); 
		\draw (18,0) -- (9,6);
		
		\foreach \x in {1,...,18}{\node at (\x,6.5) {\tiny $\x$};}
		\foreach \x in {1,...,18}{\node at (\x,-0.5) {\tiny $\x$};}
	\end{scope}
	
	\end{tikzpicture}
		\caption{Modular braids for $W_L$ (left) and $W_R$ (right).}
		\label{Lorenz-Lyndon-Markov}
	\end{figure}
\end{example}

\begin{remark}\label{rem:T-braid}
	Although a Lorenz braid is defined as an element of the braid group $B_{P_r+Q_r}$, Birman--Kofman~\cite{BirmanKofman2009} introduced the concept of a \emph{$T$-braid}, which belongs to a smaller braid group yet has the same closure. Here we briefly introduce their result. Assume $p_1 > 1$ and $q_r > 1$ as before. For each $1 \le i \le Q_r$, let $d_i \in \Z_{>1}$ be such that the $i$-th top point is connected to the $(i+d_i)$-th bottom point. This produces a vector $\vec{d} = (d_1, \dots, d_{Q_r})$. Writing $\vec{d}$ in the form
	\[
		\vec{d} = (r_1^{s_1}, \dots, r_k^{s_k})
	\]
	with $1 < r_1 < \cdots < r_k$ and $s_j > 0$, (where $r^s$ denotes $r$ repeated $s$ times), Birman and Kofman~\cite[Theorem 1]{BirmanKofman2009} showed that the braid defined by
	\[
		\mathbb{T}((r_1, s_1), \dots, (r_k, s_k)) = (\sigma_1 \sigma_2 \cdots \sigma_{r_1-1})^{s_1} (\sigma_1 \sigma_2 \cdots \sigma_{r_2-1})^{s_2} \cdots (\sigma_1 \sigma_2 \cdots \sigma_{r_k-1})^{s_k} \in B_{r_k}
	\]
	has the same closure as the original Lorenz braid. In the example of \cref{Lorenz-braid-example}, we have $\vec{d} = (2,2,2,2,3,3,6,8,8) = (2^4, 3^2, 6^1, 8^2)$, so the corresponding $T$-braid is $\mathbb{T}((2,4), (3,2), (6,1), (8,2)) \in B_8$.
\end{remark}

\begin{remark}
	For $r=1$, the condition that $p_1$ and $q_1$ are coprime is necessary and sufficient for the closure of the Lorenz braid $\mathbb{L}(p_1, q_1)$ to form a knot. In general, it seems difficult to describe the number of components of the closure of a Lorenz braid $\mathbb{L}(p_1, q_1; \dots; p_r, q_r)$ by a closed condition involving $p_j, q_j$. For example, in the case $p_1 = q_1 = \cdots = p_r = q_r = 1$, the closure of the Lorenz braid $\mathbb{L}((1,1)^r)$ defines a knot if and only if, via the Bernoulli shift, the period of the binary expansion of $1/(2r+1)$ has length $2r$. This is equivalent to requiring that the smallest positive integer $k$ satisfying $2^k \equiv 1 \pmod{2r+1}$ be $k=2r$, which in turn implies that $2r+1$ must be prime. The infinitude of such primes remains an open problem, so unlike the $r=1$ case with $\gcd(p_1, q_1) = 1$, a comparably simple criterion cannot be expected. (cf.~\emph{Artin's primitive root conjecture}~\cite{Moree2012}).
	\begin{figure}[H]
	\centering
	\begin{tikzpicture}[line width=1pt, scale=0.4]
	\begin{scope}[shift={(0,0)}]
		\draw (0,0) -- (8,0);
		\draw (0,6) -- (8,6);
		
		\def \dotPt {4pt}; 
		\foreach \x in {1,3,5,7}{\fill (\x,0) circle (\dotPt);}
		\foreach \x in {1,3,5,7}{\fill (\x,6) circle (\dotPt);}
	
		\draw (1,0) -- (5,6);
		\draw (5,0) -- (7,6);
		
		\draw[white, line width=4pt] (3-1/4,3/4) -- (1+1/4, 21/4);
		\draw (3,0) -- (1,6);
		\draw[white, line width=4pt] (7-1/4,0+3/8) -- (3+1/4,45/8); 
		\draw (7,0) -- (3,6);
		
		\foreach \x in {1,2,3,4}{\node at (2*\x-1,6.5) {\tiny $\x$};}
		\foreach \x in {1,2,3,4}{\node at (2*\x-1,-0.5) {\tiny $\x$};}
	\end{scope}
	
	\begin{scope}[shift={(12,0)}]
		\draw (0,0) -- (11,0);
		\draw (0,6) -- (11,6);
		
		\def \dotPt {4pt}; 
		\foreach \x in {1,...,10}{\fill (\x,0) circle (\dotPt);}
		\foreach \x in {1,...,10}{\fill (\x,6) circle (\dotPt);}
	
		\foreach \a/\b in {1/6, 3/7, 5/8, 7/9, 9/10}{\draw (\a,0) -- (\b,6);}
		
		\draw[white, line width=4pt] (2-1/16,3/8) -- (1+1/16,45/8); 
		\draw (2,0) -- (1,6);
		\draw[white, line width=4pt] (4-1/8,0+3/8) -- (2+1/8,45/8); 
		\draw (4,0) -- (2,6);
		\draw[white, line width=4pt] (6-1/8,1/4) -- (3+1/8,23/4); 
		\draw (6,0) -- (3,6);
		\draw[white, line width=4pt] (8-1/8,3/16) -- (4+1/8,93/16); 
		\draw (8,0) -- (4,6);
		\draw[white, line width=4pt] (10-1/4,3/10) -- (5+1/4,57/10); 
		\draw (10,0) -- (5,6);
		
		\foreach \x in {1,...,10}{\node at (\x,6.5) {\tiny $\x$};}
		\foreach \x in {1,...,10}{\node at (\x,-0.5) {\tiny $\x$};}
	\end{scope}
	
	\end{tikzpicture}
		\caption{For $r=2$, the rational number $1/5 = 0. \overline{0011}_{(2)}$ corresponds to the Lyndon word $W = LLRR$, which gives $\mathbb{L}(1,1; 1,1)$. For $r=5$, the corresponding rational is $1/11 = 0.\overline{0001011101}_{(2)}$.}
		\label{}
	\end{figure}
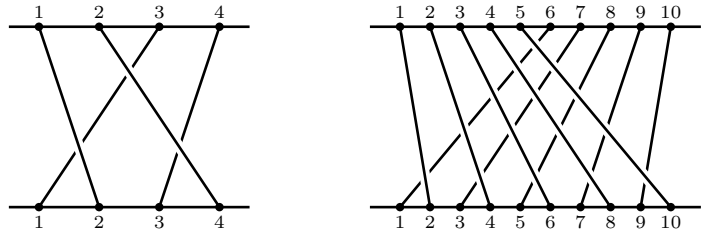
\end{remark}

\section{Torus knots as modular knots}\label{sec:Torus-knot}

As Birman--Williams~\cite{BirmanWilliams1983} showed, every torus knot of type $(p,q)$ with coprime positive integers $p$ and $q$ is a Lorenz knot. Here, the $(p,q)$-torus knot is the knot given as the closure of $\mathbb{T}(p,q) = (\sigma_1 \sigma_2 \cdots \sigma_{p-1})^q \in B_p$. From \cref{rem:T-braid}, this coincides with the closure of the Lorenz braid $\mathbb{L}(p,q)$. 
In this section, we realize $(p,q)$-torus knots as modular knots.

\subsection{Christoffel words}

For coprime positive integers $0 < p < q$, the \emph{snake graph} $S(p/q)$ is defined as follows: Consider a $p \times q$ grid of squares, and draw the diagonal from the lower left to the upper right. The snake graph is the region consisting of all unit squares intersected by this diagonal. Here is the example for $(p,q) = (4,7)$:

\begin{figure}[H]
\centering
\begin{tikzpicture}[scale=0.9,>=stealth,thick]
	\draw[thin] (0,0) grid (7,4);
	\draw (0,0) -- (7,4);

	\node[left] at (0,0) {$(0,0)$};
	\node[right] at (7,4) {$(7,4)$};

	\draw[line width=2pt, blue]
	 (0,0) -- (0,1) -- (1,1) -- (1,2) -- (3,2) -- (3,3) -- (5,3) -- (5,4) -- (7,4);

	\draw[line width=2pt,red]
	 (0,0) -- (2,0) -- (2,1) -- (4,1) -- (4,2) -- (6,2) -- (6,3) -- (7,3) -- (7,4);

	\foreach \x/\y in { 3/2, 5/3, 7/4
	}{
	  \draw[->,blue,thick, above] (\x-0.6,\y+0.2) -- ++(0.4,0);
	  \node[blue,above] at (\x-0.5,\y+0.2) {A};
	}

	\foreach \x/\y in { 0/1, 1/2, 3/3, 5/4
	}{
	  \draw[blue,thick,->] (\x-0.2,\y-0.2) -- ++(0,0.4) -- ++(0.4,0);
	\node[blue,above] at (\x+0.3,\y+0.2) {B};
	}

	\foreach \x/\y in { 1/0, 3/1, 5/2
	}{
	  \draw[->,red,thick, below] (\x-0.6,\y-0.2) -- ++(0.4,0);
	  \node[red,below] at (\x-0.5,\y-0.2) {A};
	}

	\foreach \x/\y in { 2/0, 4/1, 6/2, 7/3
	}{
	  \draw[red,thick,->] (\x-0.2,\y-0.2) -- ++(0.4,0) -- ++(0,0.4);
	  \node[red,below] at (\x-0.3,\y-0.2) {B};
	}
\end{tikzpicture}
\caption{The snake graph and its boundary for $4/7$, which are used to define the lower (red) and upper (blue) Christoffel words.}
\end{figure}
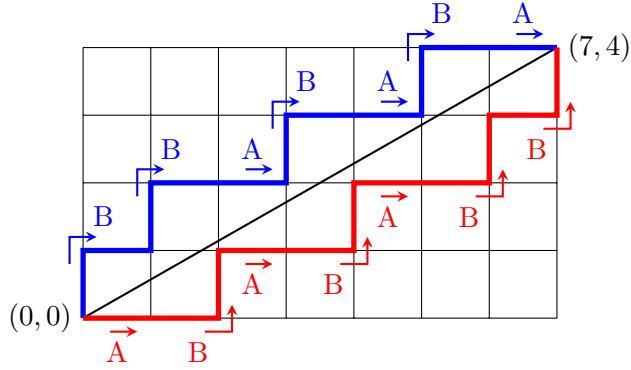

\begin{definition}
	For a snake graph $S(p/q)$, consider the ascending path along its lower boundary that follows the edges of the squares. Assign the letter $A$ each time the path moves one step to the right ($\to$), and the letter $B$ each time it moves one step to the right and one step up ($\to \uparrow$). Reading the letters along this path produces a word, called the \emph{lower Christoffel word}, which is denoted by $\mathrm{ch}_{p/q}(A,B)$.
\end{definition}

The above example provides $\mathrm{ch}_{4/7}(A,B) = ABABABB$. Similarly, the word obtained from the path along the upper boundary is called the \emph{upper Christoffel word}, which is denoted by $\mathrm{Ch}_{p/q}(A,B)$. The above example provides $\mathrm{Ch}_{4/7}(A,B) = BBABABA$. In this article, we only consider the lower Christoffel words. For further details on Christoffel words and the related Markov quadratic irrationals, see Aigner's book~\cite{Aigner2013}.

\subsection{Torus knots}\label{sec:Chris-TV}

We consider substituting $A = L$ and $B = LR$ into the lower Christoffel word. For example,
\begin{itemize}
	\item $\mathrm{ch}_{2/3}(A,B) = ABB \to \mathrm{ch}_{2/3}(L, LR) = LLRLR$.
	\item $\mathrm{ch}_{2/5}(A,B) = AABAB \to \mathrm{ch}_{2/5}(L, LR) = LLLRLLR$.
	\item $\mathrm{ch}_{3/4}(A,B) = ABBB \to \mathrm{ch}_{3/4}(L, LR) = LLRLRLR$.
	\item $\mathrm{ch}_{4/7}(A,B) = ABABABB \to \mathrm{ch}_{4/7}(L, LR) = LLRLLRLLRLR$.
\end{itemize}
This substitution yields a (Lyndon) word for each rational number in $(0,1)$. Note that this correspondence differs from the one between rational numbers and words arising from the Bernoulli shift.

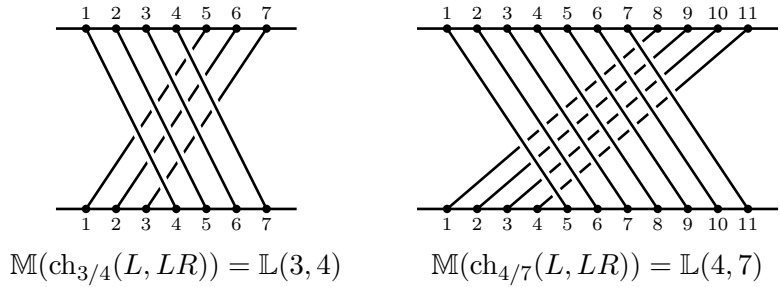
\begin{figure}[H]
\centering
\begin{tikzpicture}[line width=1pt, scale=0.4]
\begin{scope}[shift={(0,0)}]

	\def\dotPt{4pt}

	\draw (0,0) -- (8,0);
	\draw (0,6) -- (8,6);

	\foreach \x in {1,...,7}{
	  \fill (\x,0) circle (\dotPt);
	  \fill (\x,6) circle (\dotPt);
	}

	\foreach \a/\b in {1/5, 2/6, 3/7}{
	  \draw (\a,0) -- (\b,6);
	}

	\foreach \a/\b in {4/1, 5/2, 6/3, 7/4}{
	  \draw[white, line width=4pt] (\a-1/4,1/2) -- (\b+1/4,11/2);
	  \draw (\a,0) -- (\b,6);
	}

	\foreach /\x in {1,...,7}{
	  \node at (\x,6.5) {\tiny $\x$};
	  \node at (\x,-0.5) {\tiny $\x$};
	}

		\node at (4,-2) {$\mathbb{M}(\mathrm{ch}_{3/4}(L, LR)) = \mathbb{L}(3,4)$};
	\end{scope}

	\begin{scope}[shift={(12,0)}]

	\def\dotPt{4pt}

	\draw (0,0) -- (12,0);
	\draw (0,6) -- (12,6);

	\foreach \x in {1,...,11}{
	  \fill (\x,0) circle (\dotPt);
	  \fill (\x,6) circle (\dotPt);
	}

	\foreach \a/\b in {1/8, 2/9, 3/10, 4/11}{
	  \draw (\a,0) -- (\b,6);
	}

	\foreach \a/\b in {5/1, 6/2, 7/3, 8/4, 9/5, 10/6, 11/7}{
	  \draw[white, line width=4pt] (\a-1/4,3/8) -- (\b+1/4,45/8);
	  \draw (\a,0) -- (\b,6);
	}

	\foreach /\x in {1,...,11}{
	  \node at (\x,6.5) {\tiny $\x$};
	  \node at (\x,-0.5) {\tiny $\x$};
	}

	\node at (6,-2) {$\mathbb{M}(\mathrm{ch}_{4/7}(L, LR)) = \mathbb{L}(4,7)$};
\end{scope}
\end{tikzpicture}
\caption{Examples defining torus knots.}
\end{figure}

\begin{theorem}\label{thm:Torus-knot-modular}
	For coprime positive integers $0 < p < q$, we have $\mathbb{M}(\mathrm{ch}_{p/q}(L, LR)) = \mathbb{L}(p,q)$. In other words, the modular knot associated with the word $\mathrm{ch}_{p/q}(L, LR)$ is the $(p,q)$-torus knot.
\end{theorem}

\begin{proof}
	First, we show that $\mathrm{ch}_{p/q}(L, LR) = \mathrm{ch}_{p/(p+q)}(L, R)$. Note that both words contain $p$ occurrences of the letter $R$. By \cite[Proposition 7.5 (2)]{Aigner2013}, the position of the $j$-th $B$ in $\mathrm{ch}_{p/q}(A,B)$ ($1 \le j \le p$) is given by $\lceil qj/p \rceil$. Thus, in $\mathrm{ch}_{p/(p+q)}(L,R)$, the position of the $j$-th $R$ is $j+\lceil qj/p \rceil$, which agrees with that in $\mathrm{ch}_{p/q}(L,LR)$. This proves the claimed identity of the two Christoffel words.
	
	Next, by the proof of \cite[Proposition 7.27]{Aigner2013}, we have
	\[
		\mathrm{rank}_j(\mathrm{ch}_{p/(p+q)}(L, R)) \equiv p(j-1) + 1 \pmod{p+q}.
	\]
	Therefore, the modular braid for $\mathrm{ch}_{p/(p+q)}(L,R)$ coincides with the Lorenz braid $\mathbb{L}(p,q)$.
\end{proof}

\begin{example}
	The corresponding real quadratic irrationals and quadratic forms are as follows.
	\begin{align*}
		\mathrm{ch}_{2/3}(L, LR) = L^2RLR &\leftrightarrow \frac{3+2\sqrt{6}}{3} = [\overline{2,1,1,1}] \leftrightarrow 3x^2-6xy-5y^2,\\
		\mathrm{ch}_{2/5}(L, LR) = L^3RL^2R &\leftrightarrow \frac{3+2\sqrt{5}}{2} = [\overline{3,1,2,1}] \leftrightarrow 4x^2-12xy-11y^2,\\
		\mathrm{ch}_{3/4}(L, LR) = L^2RLRLR &\leftrightarrow \frac{4+\sqrt{42}}{4} = [\overline{2,1,1,1,1,1}] \leftrightarrow 8x^2-16xy-13y^2,\\
		\mathrm{ch}_{4/7}(L, LR) = L^2RL^2RL^2RLR &\leftrightarrow \frac{43+2\sqrt{1190}}{41} = [\overline{2,1,2,1,2,1,1,1}] \leftrightarrow 41x^2-86xy-71y^2.
	\end{align*}
\end{example}

\begin{example}
	For a positive integer $n$, we have $\mathrm{ch}_{n/(n+1)}(A, B) = AB^n$. The $(n,n+1)$-torus knot is represented by the word
	\[
		\mathrm{ch}_{n/(n+1)}(L, LR) = L(LR)^n.
	\]
	Let $F_n$ denote the $n$-th \emph{Fibonacci number} defined by $F_0 = F_1 = 1$ and $F_n = F_{n-1} + F_{n-2}$. The corresponding real quadratic irrational is given by
	\[
		w_n = [\overline{2, 1^{2n-1}}] = 1 + \sqrt{\frac{F_{2n+1}}{F_{2n-1}}},
	\]
	and the associated binary quadratic form is $Q_n(x,y) = F_{2n-1} x^2 + (F_{2n-2} - F_{2n+1}) xy - F_{2n} y^2$. 
	
	Since $w_n$ converges to $(3+\sqrt{5})/2 = [2,1, \overline{1,1}]$ as $n \to \infty$, we may ask whether the $(n,n+1)$-torus knot ``approaches", in some suitable sense, to the modular knot associated with the word $LR$, namely the trivial knot. This situation is reminiscent of the continuity phenomenon for integrals of the elliptic modular $j$-function along closed geodesics parametrized by Markov irrationals (related to the Christoffel words), conjectured by Kaneko~\cite{Kaneko2009} and later investigated by Bengoechea--Imamo\={g}lu~\cite{BengoecheaImamoglu2019} and Murakami~\cite{Murakami2021}.
\end{example}

\section{Some invariants of modular knots}\label{sec:topological}

\subsection{Braid index}\label{sec:braid-index}

In \cref{rem:T-braid}, we introduced the fact that for a modular braid, or more generally for a Lorenz braid, there exists an element in a smaller braid group that has the same closure. How small can the size of such a braid group be? The minimal number of strands required to represent a knot or a link as the closure of a braid is called the \emph{braid index}. For a Lorenz braid, it is known that the braid index can be computed using the following trip number.

\begin{definition}
	For a Lorenz braid $\mathbb{L} = \mathbb{L}(p_1, q_1; \dots; p_r, q_r)$, we assign the letter $L$ to each $1 \le j \le Q_r$ and the letter $R$ to each $Q_r < j \le P_r + Q_r$. The number of strands connecting an $L$ on the top to an $R$ on the bottom is called the \emph{trip number} of $\mathbb{L}$, denoted by $\mathrm{trip}(\mathbb{L})$.
\end{definition}

As shown in~\cite[Lemma 1]{BirmanKofman2009}, the trip number coincides with the braid index of the Lorenz link $\widehat{\mathbb{L}}$, which was conjectured in~\cite{BirmanWilliams1983} and first proved in~\cite{FranksWilliams1987}. For a modular braid $\mathbb{M}(W)$, the trip number equals the number of subword $LR$ contained in the word $W$. Equivalently, it is equal to half the length of the period of the continued fraction expansion of the corresponding real quadratic irrational. We restate this in the form of a theorem below.

\begin{theorem}
	For a real quadratic irrational $w = [\overline{c_1, c_2, \dots, c_{2\ell}}]$ with minimal even period $2\ell$, the braid index of the modular knot $\widehat{\mathbb{M}}(w)$ is given by $\ell$.
\end{theorem}

\begin{example}\label{ex:trip-number}
	The trip number of the Lorenz braid $\mathbb{L}$ shown in~\cref{Lorenz-braid-example} corresponds to the three thickest strands and is therefore $\mathrm{trip}(\mathbb{L}) = 3$. As seen in \eqref{eq:Lorenz-LR-word} and \eqref{eq:Lorenz-cont-frac}, the corresponding word is $W = L^4 R^3 L^4 R^3 LR^2$ and the associated real quadratic irrational is $w = [\overline{4,3,4,3,1,2}]$. Both the number of subwords $LR$ and half the length of the period of its continued fraction expansion are equal to $3$. 
\end{example}

\subsection{Class numbers of imaginary quadratic fields}

Let $p > 3$ be a prime number satisfying $p \equiv 3 \pmod{4}$. Consider the continued fraction expansion
\[
	\sqrt{p} = [c_0, \overline{c_1, c_2, \dots, c_{2\ell}}],
\]
which corresponds to the word $L^{c_2} R^{c_3} \cdots L^{c_{2\ell}} R^{c_1}$. We then consider the modular braid $\mathbb{M}(\sqrt{p})$ obtained from this word.

\begin{example}
	For
	\[
		\sqrt{7} = [2, \overline{1,1,1,4}],
	\]
	we consider the word $W = LRL^4R$. Its cyclic permutations are given by
	\[
	\begin{array}{c|ccccccc}
		\mathrm{cyc}^{\circ (j-1)}(W) & LRL^4R & RL^4RL & L^4RLR & L^3RLRL & L^2RLRL^2 & LRLRL^3 & RLRL^4\\
		\hline
		\mathrm{rank}_j(W) & 4 & 6 & 1 & 2 & 3 & 5 & 7
	\end{array}
	\]
	The corresponding braid $\mathbb{M}(\sqrt{7})$ is given as follows. Similarly, the modular braid for $\sqrt{31} = [5, \overline{1,1,3,5,3,1,1,10}]$ is given as follows.
	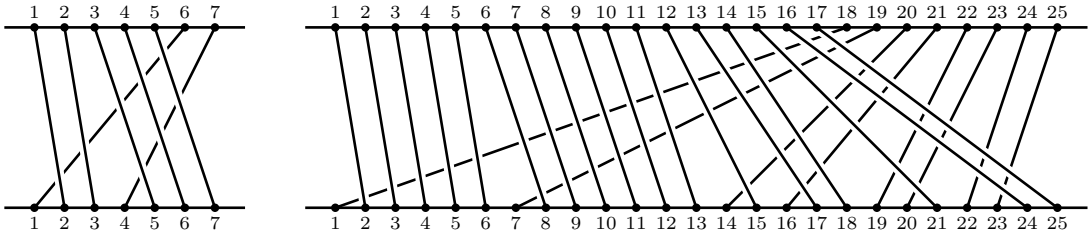
\begin{figure}[H]
	\centering
	\begin{tikzpicture}[line width=1pt, scale=0.4]
	\begin{scope}[shift={(0,0)}]
		\draw (0,0) -- (8,0);
		\draw (0,6) -- (8,6);
	
		\def \dotPt {4pt}; 
		\foreach \x in {1,...,7}{\fill (\x,0) circle (\dotPt);}
		\foreach \x in {1,...,7}{\fill (\x,6) circle (\dotPt);}
	
		\draw (1,0) -- (6,6);
		\draw (4,0) -- (7,6);
		
		\draw[white, line width=4pt] (2-1/16,3/8) -- (1+1/16, 45/8);
		\draw (2,0) -- (1,6);
		\draw[white, line width=4pt] (3-1/16,0+3/8) -- (2+1/16,45/8); 
		\draw (3,0) -- (2,6);
		\draw[white, line width=4pt] (5-1/8,0+3/8) -- (3+1/8,45/8); 
		\draw (5,0) -- (3,6);
		\draw[white, line width=4pt] (6-1/8,0+3/8) -- (4+1/8,45/8); 
		\draw (6,0) -- (4,6);
		\draw[white, line width=4pt] (7-1/8,0+3/8) -- (5+1/8,45/8); 
		\draw (7,0) -- (5,6);
		
		\foreach \x in {1,...,7}{\node at (\x,6.5) {\tiny $\x$};}
		\foreach \x in {1,...,7}{\node at (\x,-0.5) {\tiny $\x$};}
	\end{scope}

	\begin{scope}[shift={(10,0)}]
		
		\def \dotPt {4pt}; 
		\foreach \x in {1,...,25}{\fill (\x,0) circle (\dotPt);}
		\foreach \x in {1,...,25}{\fill (\x,6) circle (\dotPt);}
	
		\draw (1,0) -- (18,6);
		\draw (7,0) -- (19,6);
		\draw (14,0) -- (20,6);
		\draw (16,0) -- (21,6);
		\draw (19,0) -- (22,6);
		\draw (20,0) -- (23,6);
		\draw (22,0) -- (24,6);
		\draw (23,0) -- (25,6);
		
		\draw[white, line width=4pt] (2-1/32,0+3/16) -- (1+1/32,93/16); 
		\draw (2,0) -- (1,6);
		\draw[white, line width=4pt] (3-1/32,0+3/16) -- (2+1/32,93/16); 
		\draw (3,0) -- (2,6);
		\draw[white, line width=4pt] (4-1/32,0+3/16) -- (3+1/32,93/16); 
		\draw (4,0) -- (3,6);
		\draw[white, line width=4pt] (5-1/32,0+3/16) -- (4+1/32,93/16); 
		\draw (5,0) -- (4,6);
		\draw[white, line width=4pt] (6-1/32,0+3/16) -- (5+1/32,93/16); 
		\draw (6,0) -- (5,6);
		\draw[white, line width=4pt] (8-1/16,0+3/16) -- (6+1/16,93/16); 
		\draw (8,0) -- (6,6);
		\draw[white, line width=4pt] (9-1/16,0+3/16) -- (7+1/16,93/16); 
		\draw (9,0) -- (7,6);
		\draw[white, line width=4pt] (10-1/16,0+3/16) -- (8+1/16,93/16); 
		\draw (10,0) -- (8,6);
		\draw[white, line width=4pt] (11-1/16,0+3/16) -- (9+1/16,93/16); 
		\draw (11,0) -- (9,6);
		\draw[white, line width=4pt] (12-1/16,0+3/16) -- (10+1/16,93/16); 
		\draw (12,0) -- (10,6);
		\draw[white, line width=4pt] (13-1/16,0+3/16) -- (11+1/16,93/16); 
		\draw (13,0) -- (11,6);
		\draw[white, line width=4pt] (15-1/8,1/4) -- (12+1/8,23/4); 
		\draw (15,0) -- (12,6);
		\draw[white, line width=4pt] (17-1/8,3/16) -- (13+1/8,93/16); 
		\draw (17,0) -- (13,6);
		\draw[white, line width=4pt] (18-1/8,3/16) -- (14+1/8,93/16); 
		\draw (18,0) -- (14,6);
		\draw[white, line width=4pt] (21-1/4,1/4) -- (15+1/4,23/4); 
		\draw (21,0) -- (15,6);
		\draw[white, line width=4pt] (24-1/3,1/4) -- (16+1/3,23/4); 
		\draw (24,0) -- (16,6);
		\draw[white, line width=4pt] (25-1/3,1/4) -- (17+1/8,189/32); 
		\draw (25,0) -- (17,6);
		\foreach \x in {1,...,25}{\node at (\x,6.5) {\tiny $\x$};}
		\foreach \x in {1,...,25}{\node at (\x,-0.5) {\tiny $\x$};}
		
		\draw (0,0) -- (26,0);
		\draw (0,6) -- (26,6);
	\end{scope}

	\end{tikzpicture}
		\caption{Modular braids $\mathbb{M}(\sqrt{7}) = \mathbb{L}(1,2; 1,3)$ and $\mathbb{M}(\sqrt{31}) = \mathbb{L}(1,5;1,6;1,1;1,2;2,1;2,2)$.}
		\label{Lorenz-7}
	\end{figure}
\end{example}

Then the result concerning the class number $h(-p)$ of the imaginary quadratic field $\Q(\sqrt{-p})$, due to Hirzebruch and Zagier, can be restated as follows.

\begin{theorem}[\cite{Hirzebruch1973, Zagier1975}]\label{thm:HZ}
	For a prime $3 < p \equiv 3 \pmod{4}$, write $\mathbb{M}(\sqrt{p}) = \mathbb{L}(p_1, q_1; \dots; p_r, q_r)$. If the class number of the real quadratic field $\Q(\sqrt{p})$ in wide sense is equal to 1, then we have
	\[
		h(-p) = \frac{1}{3} (Q_r - P_r),
	\]
	where $P_r = p_1 + \cdots + p_r, Q_r = q_1 + \cdots + q_r$.
\end{theorem}

\begin{proof}
	Since $Q_r$ is the number of $L$'s and $P_r$ is the number of $R$'s in the word representation, we can express
	\[
		Q_r - P_r = \sum_{j=1}^{2\ell} (-1)^j c_j
	\]
	using the coefficients of the continued fraction expansion of $\sqrt{p}$. Comparing this with the result by Hirzebruch~\cite{Hirzebruch1973} and Zagier~\cite{Zagier1975},
	\[
		h(-p) = \frac{1}{3} \sum_{j=1}^{2\ell} (-1)^j c_j,
	\]
	we obtain the desired result.
\end{proof}

For the above examples, (using the fact that the class numbers of $\Q(\sqrt{7})$ and $\Q(\sqrt{31})$ are both 1), we obtain
\begin{align*}
	h(-7) = \frac{5-2}{3} = 1,\qquad h(-31) = \frac{17 - 8}{3} = 3.
\end{align*}

We make a further remark on this result. For a real quadratic irrational $w = [\overline{c_1, c_2, \dots, c_{2\ell}}]$, by associating the matrix
\[
	\gamma_w = T^{c_1} V^{c_2} \cdots V^{c_{2\ell}} \in \SL_2(\Z)
\]
in the same way as in the proof of~\cref{lem:cont-word-corr}, Rademacher~\cite{Rademacher1956, Rademacher 1972} proved the following. Let $\log \Delta(z)$ denote the logarithm of the discriminant cusp form $\Delta(z)$, (essentially the same as the Dedekind eta function), defined by
\[
	\log \Delta(z) = 2\pi iz - 24\sum_{n=1}^\infty \sum_{d \mid n} d^{-1} e^{2\pi inz} \qquad (z \in \bbH).
\]
Then the modular transformation under the action of $\gamma_w$ is given by
\[
	\log \Delta(\gamma_w z) - \log \Delta(z) = 12 \log j(\gamma_w, z) + 2\pi i \Psi(\gamma_w),
\]
where $j(\smat{a & b \\ c & d}, z) = cz+d$ is the usual automorphic factor (assuming $\Im \log z \in (-\pi, \pi)$), and $\Psi(\gamma_w)$ is given by the alternating sum
\[
	\Psi(\gamma_w) = \sum_{j=1}^{2\ell} (-1)^{j-1} c_j.
\]
In general, we can define the function $\Psi: \SL_2(\Z) \to \Z$ on the whole of $\SL_2(\Z)$, called the \emph{Rademacher symbol}, (see also Atiyah's omnibus theorem~\cite[Theorem 5.60]{Atiyah1987}). Ghys~\cite{Ghys2006} characterized modular knots in terms of the geodesic flow, and showed that when realized in the trefoil complement $\SL_2(\Z) \backslash \SL_2(\R) \cong S^3 - K_{2,3}$, the above alternating sum coincides with the linking number $\mathrm{link}(\widehat{\mathbb{M}}(w), K_{2,3})$. By comparing this with the expression of the Rademacher symbol described above, he established the identity $\mathrm{link}(\widehat{\mathbb{M}}(w), K_{2,3}) = \Psi(\gamma_w)$.

\begin{remark}
	In the case $p \equiv 1 \pmod{4}$, Kaneko--Mizuno~\cite{KanekoMizuno2020} gave several analogous formulas, which can also be reformulated in terms of modular braids.
\end{remark}

\section{Alexander polynomials}\label{sec:Alexander}

The main purpose of this section is to review the definition of the Alexander polynomial using the Burau representation (\cref{def:Alexander-poly}) and to compute the Burau representation explicitly for Lorenz braids (\cref{thm:Burau-explicit}).

\subsection{Burau representations}\label{sec:Burau}

Following Burau's result~\cite{Burau1935}, (see also Birman~\cite[Theorem 3.11]{Birman1974}, Kassel--Turaev~\cite[Theorem 3.13]{KasselTuraev2008}), we define the Alexander polynomial. Here, we present two equivalent definitions, one using the Burau representation itself, and the other using its irreducible part.

\begin{definition}\label{def:Burau-rep}
	The \emph{Burau representation} $\beta_n: B_n \to \GL_n(\Z[t, t^{-1}])$ is defined by
	\[
		\beta_n(\sigma_j) = I_{j-1} \oplus \pmat{1-t & t \\ 1 & 0} \oplus I_{n-j-1}.
	\]
\end{definition}

For any $\sigma \in B_n$, we have $\beta_n(\sigma) v = v$, where $v = {}^t(1,1,\dots, 1)$, that is, $\beta_n(\sigma)$ has the eigenvalue $1$. In particular, the Burau representation is reducible, and we often consider the following $(n-1)$-dimensional irreducible representation $\beta_n^r: B_n \to \GL_{n-1}(\Z[t,t^{-1}])$, called the \emph{reduced Burau representation}:
\[
	\beta_n^r(\sigma_j) = \begin{cases}
		\pmat{-t & 0 \\ 1 & 1} \oplus I_{n-3} &\text{if } j=1,\\
		I_{j-2} \oplus \pmat{1 & t & 0 \\ 0 & -t & 0 \\ 0 & 1 & 1} \oplus I_{n-j-2} &\text{if } 1 < j < n-1,\\
		I_{n-3} \oplus \pmat{1 & t \\ 0 & -t} &\text{if } j=n-1,
	\end{cases}
\]
(see \cite[Theorem 3.9]{KasselTuraev2008} for instance).

\begin{definition}\label{def:Alexander-poly}
	For a braid $\sigma \in B_n$, we assume that its closure $\widehat{\sigma}$ defines a knot. The \emph{Alexander polynomial} $\Delta_{\widehat{\sigma}}(t) \in \Z[t,t^{-1}]$ of a knot $\widehat{\sigma}$ is defined by
	\begin{align}\label{def:Alexander-second}
		\Delta_{\widehat{\sigma}}(t) &= \frac{1}{[n]_t} \lim_{s \to 1} \frac{\det(s I_n - \beta_n(\sigma))}{s-1} \nonumber\\
			&= \frac{1}{[n]_t} \det(I_{n-1} - \beta_n^r(\sigma)),
	\end{align}
	where $[n]_t \coloneqq 1 + t + t^2 + \cdots + t^{n-1}$.
\end{definition}

In this section, we focus on the representation $\beta_n$ rather than on its reduced version.

\begin{remark}\label{rem:Alex-def}
	It is well known that our definition of the Alexander polynomial using the Burau representation is equivalent to other standard definitions of the Alexander polynomial up to multiplication by units in $\Z[t,t^{-1}]$. 
\end{remark}

\begin{remark}
	Okamoto~\cite[Theorem 2.1.1]{Okamoto2018} investigated the function $\det(I_n - \beta_n(\sigma) s)^{-1}$, which he referred to as the \emph{braid zeta function}. He also studied the Jones polynomials, the HOMFLY polynomials, and other invariants from the viewpoint of zeta functions.
\end{remark}

The following symmetry is a basic property of the Alexander polynomials.

\begin{lemma}\label{lem:sym-Alexander}
	For $\sigma \in B_n$ whose closure is a knot, let $|\sigma| \in \Z$ denote the sum of the exponents when $\sigma$ is expressed as a product of the generators, so that $\det(\beta_n(\sigma)) = (-t)^{|\sigma|}$. Then, we have
	\[
		\Delta_{\widehat{\sigma}}(1/t) = (-t)^{-|\sigma|+n-1} \Delta_{\widehat{\sigma}}(t).
	\]
\end{lemma}

\begin{proof}
	Let $\Omega_n$ be the $n \times n$ lower triangular matrix whose diagonal entries are $1$ and all strictly lower-triangular entries are $1-t$. That is,
	\[
		\Omega_n = \pmat{1 & 0 & \cdots & 0 \\ 1-t & 1 & \ddots & \vdots \\ \vdots & \ddots & \ddots & 0 \\ 1-t & \cdots & 1-t & 1} \in \GL_n(\Z[t,t^{-1}]).
	\]
	Then, it is known that
	\[
		\beta_n(\sigma) \bigg|_{t \to 1/t} = \Omega_n {}^t \beta_n(\sigma^{-1}) \Omega_n^{-1},
	\]
	(see \cite[Theorem 3.1]{KasselTuraev2008}). Therefore, we obtain
	\begin{align*}
		\det(sI_n - \beta_n(\sigma)) \bigg|_{t \to 1/t} &= \det(sI_n - {}^t \beta_n(\sigma^{-1}))\\
			&= (-s)^n \det(\beta_n(\sigma))^{-1} \det(s^{-1}I_n - \beta_n(\sigma)).
	\end{align*}
	This implies the desired result.
\end{proof}

In this article, the Alexander polynomial is defined only for knots. Whether a braid $\sigma \in B_n$ determines a knot can be naturally reformulated in terms of the Burau representation as follows.

Let $\mathfrak{S}_n$ be the symmetric group of degree $n$ and $\pi: B_n \to \mathfrak{S}_n$ denote the natural projection defined by $\pi(\sigma_j) = (j, j+1)$. By definition, the Burau representation evaluated at $t=1$ is reduced to the permutation representation $p_n: \mathfrak{S}_n \to \GL_n(\Z)$, that is,
\[
	\lim_{t \to 1} \beta_n(\sigma) = p_n(\pi(\sigma))
\]
holds. 

\begin{lemma}
	For a permutation $\pi \in \mathfrak{S}_n$, we have
	\[
		\det(s I_n - p_n(\pi)) = \prod_{P \in \mathrm{Cycle}(\pi)} (s^{l(P)} - 1),
	\]
	where $\mathrm{Cycle}(\pi)$ denotes the set of cycles in a decomposition of $\pi$ into disjoint cyclic permutations, and $l(P)$ denotes the length of a cycle $P$.
\end{lemma}

\begin{proof}
	Let $\pi \in \mathfrak{S}_n$ be written as a product of disjoint cycles $\pi = P_1 \cdots P_r$, and let $l_j = l(P_j)$ denote the length of each cycle. Then there exists a suitable $\rho \in \mathfrak{S}_n$ such that
	\[
		\rho^{-1} \pi \rho = (1, \dots, l_1) (l_1+1, \dots, l_1+l_2) \cdots (l_1 + \cdots + l_{r-1} + 1, \dots, n),
	\]
	which is a product of cycles consisting of consecutive integers. It follows that
	\[
		p_n(\rho^{-1} \pi \rho) = \mathrm{diag}(C_{l_1}, \dots, C_{l_r}),
	\]
	where each $C_l$ is an $l \times l$ matrix of the form
	\[
		C_l = \pmat{0 & 1 &  &  \\ & 0 & \ddots \\  & & \ddots & 1 \\ 1 & & & 0}.
	\]
	Since the determinant is invariant under conjugation, we obtain
	\begin{align*}
		\det(s I_n - p_n(\pi)) &= \det(s I_n - \mathrm{diag}(C_{l_1}, \dots, C_{l_r}))\\
			&= \prod_{j=1}^r \det(sI_{l_j} - C_{l_j}) = \prod_{j=1}^r (s^{l_j} - 1),
	\end{align*}
	which proves the claim.
\end{proof}

\begin{corollary}\label{cor:component-number}
	For a braid $\sigma \in B_n$, the number of components of its closure $\widehat{\sigma}$ equals the multiplicity of the eigenvalue $1$ of the matrix 
	\[
		\lim_{t \to 1} \beta_n(\sigma) = p_n(\pi(\sigma)).
	\]
\end{corollary}

\begin{example}
	Since a braid given in \cref{Lorenz-Lyndon-ababb} has the expression $\sigma = \sigma_2 \sigma_1 \sigma_3 \sigma_2 \sigma_4 \sigma_3 \in B_5$, 
	\begin{figure}[H]
	\centering
	\begin{tikzpicture}[line width=1pt, scale=0.3]
	\begin{scope}[shift={(0,0)}]

		\def\dotPt{4pt}

		\draw (0,0) -- (10,0);
		\draw (0,4) -- (10,4);

		\foreach \x in {1,3,5,7,9}{
		  \fill (\x,0) circle (\dotPt);
		  \fill (\x,4) circle (\dotPt);
		}

		\foreach \a/\b in {1/5, 3/7, 5/9}{
		  \draw (\a,0) -- (\b,4);
		}

			\foreach \a/\b in {7/1, 9/3}{
		  \draw[white, line width=5pt] (\a-0.5,0+1/3) -- (\b+0.5,11/3);
		  \draw (\a,0) -- (\b,4);
		}

		\foreach \i/\x in {1/1, 2/3, 3/5, 4/7, 5/9}{
		  \node at (\x,5) {$\i$};
		  \node at (\x,-1) {$\i$};
		}

		\node at (11.5,2) {$=$};
	\end{scope}

	\begin{scope}[shift={(13,-4)}]

	\def\dotPt{4pt}

	\draw (1,0) -- (1,8);
	\draw (1,10) -- (1,12);
	\draw (3,0) -- (3,4);
	\draw (3,6) -- (3,8);
	\draw (5,2) -- (5,4);
	\draw (5,8) -- (5,10);
	\draw (7,4) -- (7,6);
	\draw (7,8) -- (7,12);
	\draw (9,0) -- (9,2);
	\draw (9,4) -- (9,12);
	\draw (1,8) -- (5,12);
	\draw (3,4) -- (7,8);
	\draw (5,0) -- (9,4);
	
	\foreach \x/\y/\xx/\yy in {
	  1/10/3/8,
	  3/6/5/4,
	  3/12/5/10,
	  5/2/7/0,
	  5/8/7/6,
	  7/4/9/2
	}{
	  \draw[white, line width=4pt] (\x,\y) -- (\xx,\yy);
	  \draw (\x,\y) -- (\xx,\yy);
	}
	
	\foreach \y in {0,2,...,12}{
	  \draw (0,\y) -- (10,\y);
	}
	
	\foreach \x in {1,3,5,7,9}{
	  \foreach \y in {0,2,...,12}{
	    \fill (\x,\y) circle (\dotPt);
	  }
	}
	
	\foreach \i/\x in {1/1, 2/3, 3/5, 4/7, 5/9}{
	  \node at (\x,13) {$\i$};
	  \node at (\x,-1) {$\i$};
	}
	
	\node at (11,11) {$\sigma_2$};
	\node at (11,9) {$\sigma_1$};
	\node at (11,7) {$\sigma_3$};
	\node at (11,5) {$\sigma_2$};
	\node at (11,3) {$\sigma_4$};
	\node at (11,1) {$\sigma_3$};
	\end{scope}

	\end{tikzpicture}
		\caption{The braid $\sigma = \sigma_2 \sigma_1 \sigma_3 \sigma_2 \sigma_4 \sigma_3 \in B_5$}
		\label{Trefoil-Lorenz-braid}
	\end{figure}
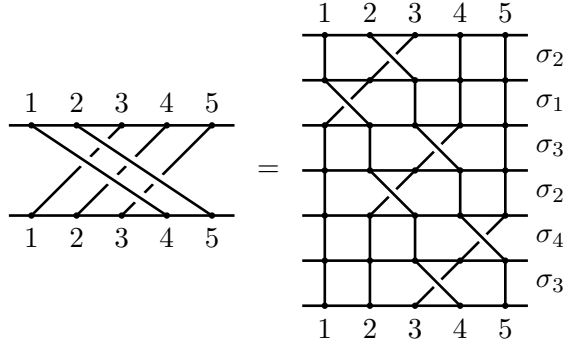
	\noindent the image under the Burau representation is given by
	\[
		\beta_5(\sigma) = \pmat{1-t & (1-t)t & (1-t)t^2 & t^3 & 0 \\ 1-t & (1-t)t & (1-t)t^2 & 0 & t^3 \\ 1 & 0 & 0 & 0 & 0 \\ 0 & 1 & 0 & 0 & 0 \\ 0 & 0 & 1 & 0 & 0}.
	\]
	Then
	\[
		\det(sI_5 - \beta_5(\sigma)) = (s-1)(s^4 + t^2 s^3 + t^3 s^2 + t^4 s + t^6),
	\]
	and hence
	\[
		\Delta_{\widehat{\sigma}}(t) = \frac{1}{[5]_t} \lim_{s \to 1} \frac{\det(sI_5 - \beta_5(\sigma))}{s-1} = t^2 - t + 1.
	\]
	At $t=1$, the characteristic polynomial of $p_n(\pi(\sigma))$ is $(s-1)(s^4+s^3+s^2+s+1) = s^5-1$. Since the eigenvalue $1$ has multiplicity $1$, the closure $\widehat{\sigma}$ has a single component. Indeed, it is the trefoil.
\end{example}

\begin{example}\label{ex:Alex-Lorenz}
	For the braid $\sigma = \mathbb{L}(2,4; 1,2; 3,1; 2,2)$ shown in \cref{Lorenz-braid-example}, 
	it can be expressed as
	\begin{align*}
		\sigma &= \sigma_9 \sigma_8 \sigma_{10} \sigma_7 \sigma_9 \sigma_{11} \sigma_6 \sigma_8 \sigma_{10} \sigma_{12} \sigma_5 \sigma_7 \sigma_9 \sigma_{11} \sigma_{13} \sigma_4 \sigma_6 \sigma_8 \sigma_{10} \sigma_{12} \sigma_{14} \\
		&\qquad \cdot \sigma_3 \sigma_5 \sigma_7 \sigma_{11} \sigma_{13} \sigma_{15} \sigma_2 \sigma_4 \sigma_{12} \sigma_{14} \sigma_{16} \sigma_1 \sigma_3 \sigma_{15} \sigma_2 \in B_{17}.
	\end{align*}
	Its image under the Burau representation is
	\begin{align*}
		&\beta_{17}(\sigma) \\
		&= \begin{tiny}
		\left(
\begin{array}{cc|cccc|c|cc|ccc|c|cc|cc}
 1-t & (1-t) t & t^2 & 0 & 0 & 0 & 0 & 0 & 0 & 0 & 0 & 0 & 0 & 0 & 0 & 0 & 0 \\
 1-t & (1-t) t & 0 & t^2 & 0 & 0 & 0 & 0 & 0 & 0 & 0 & 0 & 0 & 0 & 0 & 0 & 0 \\
 1-t & (1-t) t & 0 & 0 & t^2 & 0 & 0 & 0 & 0 & 0 & 0 & 0 & 0 & 0 & 0 & 0 & 0 \\
 1-t & (1-t) t & 0 & 0 & 0 & t^2 & 0 & 0 & 0 & 0 & 0 & 0 & 0 & 0 & 0 & 0 & 0 \\ \hline
 1-t & (1-t) t & 0 & 0 & 0 & 0 & (1-t) t^2 & t^3 & 0 & 0 & 0 & 0 & 0 & 0 & 0 & 0 & 0 \\
 1-t & (1-t) t & 0 & 0 & 0 & 0 & (1-t) t^2 & 0 & t^3 & 0 & 0 & 0 & 0 & 0 & 0 & 0 & 0 \\ \hline
 1-t & (1-t) t & 0 & 0 & 0 & 0 & (1-t) t^2 & 0 & 0 & (1-t) t^3 & (1-t) t^4 & (1-t) t^5 & t^6 & 0 & 0 & 0 & 0 \\ \hline
 1-t & (1-t) t & 0 & 0 & 0 & 0 & (1-t) t^2 & 0 & 0 & (1-t) t^3 & (1-t) t^4 & (1-t) t^5 & 0 & (1-t) t^6 & (1-t) t^7 & t^8 & 0 \\
 1-t & (1-t) t & 0 & 0 & 0 & 0 & (1-t) t^2 & 0 & 0 & (1-t) t^3 & (1-t) t^4 & (1-t) t^5 & 0 & (1-t) t^6 & (1-t) t^7 & 0 & t^8 \\ \hline
 1 & 0 & 0 & 0 & 0 & 0 & 0 & 0 & 0 & 0 & 0 & 0 & 0 & 0 & 0 & 0 & 0 \\
 0 & 1 & 0 & 0 & 0 & 0 & 0 & 0 & 0 & 0 & 0 & 0 & 0 & 0 & 0 & 0 & 0 \\ \hline
 0 & 0 & 0 & 0 & 0 & 0 & 1 & 0 & 0 & 0 & 0 & 0 & 0 & 0 & 0 & 0 & 0 \\ \hline
 0 & 0 & 0 & 0 & 0 & 0 & 0 & 0 & 0 & 1 & 0 & 0 & 0 & 0 & 0 & 0 & 0 \\
 0 & 0 & 0 & 0 & 0 & 0 & 0 & 0 & 0 & 0 & 1 & 0 & 0 & 0 & 0 & 0 & 0 \\
 0 & 0 & 0 & 0 & 0 & 0 & 0 & 0 & 0 & 0 & 0 & 1 & 0 & 0 & 0 & 0 & 0 \\ \hline
 0 & 0 & 0 & 0 & 0 & 0 & 0 & 0 & 0 & 0 & 0 & 0 & 0 & 1 & 0 & 0 & 0 \\
 0 & 0 & 0 & 0 & 0 & 0 & 0 & 0 & 0 & 0 & 0 & 0 & 0 & 0 & 1 & 0 & 0 \\
\end{array}
\right).
\end{tiny}
	\end{align*}
	Then we have
	\begin{align*}
		\det(sI_{17} - \beta_{17}(\sigma)) &= (s-1) (s^{16}+s^{15} t^3+s^{14} t^6+s^{13} t^8+s^{12} t^{10}+s^{11} t^{12}+s^{10} t^{14}+s^9 t^{15}+s^8 t^{18}\\
			&\qquad +s^7 t^{21}+s^6 t^{22}+s^5 t^{24}+s^4 t^{26}+s^3 t^{28}+s^2 t^{30}+s t^{33}+t^{36} ),
	\end{align*}
	which implies
	\[
		\lim_{t \to 1} \det(s I_{17} - \beta_{17}(\sigma)) = s^{17}-1,
	\]
	and 
	\begin{align*}
		\Delta_{\widehat{\sigma}}(t) &= t^{20}-t^{19}+t^{17}-t^{16}+t^{14}-t^{13}+t^{12}-t^{11}+t^{10}-t^9+t^8-t^7+t^6-t^4+t^3-t+1\\
			&= (t^2 - t + 1)(t^{18} - t^{16} + t^{12} - t^9 + t^6 - t^2 + 1).
	\end{align*}
\end{example}

\subsection{Explicit formulas for the Alexander polynomials}\label{sec-Lorenz-braid}

In this subsection, we consider general Lorenz braids. The Burau representation can be computed explicitly as follows, which allows the Alexander polynomials to be computed explicitly as well.

\begin{theorem}\label{thm:Burau-explicit}
	For a Lorenz braid $\sigma = \mathbb{L}(p_1, q_1; \dots; p_r, q_r)$, we have
	\[
		\beta_{P_r + Q_r}(\sigma) = 
		\left(\begin{array}{cc|cc|cc|cc}
		A_{q_1, p_1} & t^{P_1} I_{q_1} & 0 & 0 & &  & 0 & 0 \\
		A_{q_2, p_1} & 0 & t^{P_1}A_{q_2, p_2} & t^{P_2} I_{q_2} & & & \vdots & \vdots \\
		A_{q_3, p_1} & 0 & t^{P_1}A_{q_3, p_2} & 0 & \cdots & \cdots & \vdots & \vdots\\
		\vdots & \vdots & \vdots & \vdots & & & 0 & 0\\
		A_{q_r, p_1} & 0 & t^{P_1}A_{q_r, p_2} & 0 & & & t^{P_{r-1}} A_{q_r, p_r} & t^{P_r} I_{q_r} \\ \hline
		I_{p_1} & 0 & 0 & 0 & & & 0 & 0\\ 
		0 & 0 & I_{p_2} & 0 & & & 0 & 0\\
		\vdots & \vdots & \vdots & \vdots & \cdots & \cdots & \vdots & \vdots\\
		0 & 0 & 0 & 0& & & I_{p_r} & 0
		\end{array} \right),
	\]
	where $A_{q,p}$ is of size $q \times p$ defined by
	\begin{align*}
		A_{q, p} &= (1-t) \pmat{1 & t & \cdots & t^{p-1} \\ \vdots & \vdots & &\vdots \\ 1 & t & \cdots & t^{p-1}}.
	\end{align*}
\end{theorem}

\begin{proof}
	We prove this by induction on $Q_r$. When $Q_r = 1$, we have $r=1$. In this case, since $\mathbb{L}(p,1) = \sigma_1 \sigma_2 \cdots \sigma_p \in B_{p+1}$, we see that
	\begin{align*}
		\beta_{p+1}(\mathbb{L}(p,1)) &= \pmat{1-t & t &  \\ 1 & 0 & \\ & & I_{p-1}} \pmat{1 & & & \\ & 1-t & t & \\ & 1 & 0 & \\ & & & I_{p-2}} \cdots \pmat{I_{p-1} & & \\ & 1-t & t \\ & 1 & 0}
		= \pmat{A_{1,p} & t^p \\ I_p & 0}.
	\end{align*}
	This can be verified, for example, by induction on $p$. 
	
	For $Q_r > 1$, by separating the crossings involving the $Q_r$-th strand from the remaining part (see the figure below), we can express
	\begin{align*}
		&\beta_{P_r+Q_r}(\mathbb{L}(p_1, q_1; \dots; p_r, q_r))\\
		&= \begin{cases}
			I_{Q_r-1} \oplus \beta_{P_r+1}(\mathbb{L}(P_r, 1)) \cdot \beta_{P_r+Q_r-1}(\mathbb{L}(p_1, q_1; \dots; p_r, q_r-1)) \oplus I_1 &\text{if } q_r > 1,\\
			I_{Q_r-1} \oplus \beta_{P_r+1}(\mathbb{L}(P_r, 1)) \cdot \beta_{P_{r-1}+Q_{r-1}}(\mathbb{L}(p_1, q_1; \dots; p_{r-1}, q_{r-1})) \oplus I_{p_r+1} &\text{if } q_r = 1.
		\end{cases}
	\end{align*}
	
	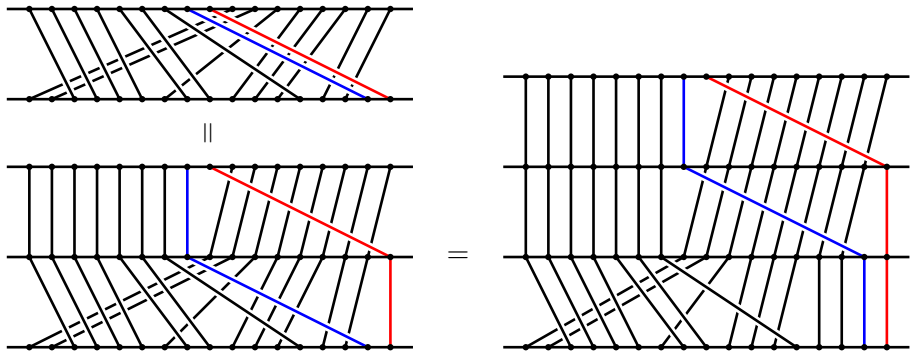
\begin{figure}[H]
	\centering
	\begin{tikzpicture}[line width=1pt, scale=0.3]
	\begin{scope}[shift={(0,0)}]
	\def\dotPt{4pt}
	
	\foreach \x/\y in {
	  1/10, 2/11, 7/12, 10/13, 11/14, 12/15, 14/16, 15/17
	}{
	  \draw (\x,0) -- (\y,4);
	}
	
	\foreach \x/\y in {
	  1/3, 2/4, 3/5, 4/6, 5/8, 6/9, 7/13
	}{
	  \draw[white, line width=3pt] (\x,4) -- (\y,0);
	  \draw (\x,4) -- (\y,0);
	}
	
	 \draw[white, line width=3pt] (8,4) -- (16,0);
	 \draw[blue] (8,4) -- (16,0);
	 \draw[white, line width=3pt] (9,4) -- (17,0);
	 \draw[red] (9,4) -- (17,0);
	
	\foreach \y in {0,4}{
	  \draw (0,\y) -- (18,\y);
	}
	
	\foreach \x in {1,...,17}{
	  \foreach \y in {0,4}{
	    \fill (\x,\y) circle (\dotPt);
	  }
	}
	
	\node at (9,-1.5) {\rotatebox{90}{$=$}};
	\end{scope}
	
	\begin{scope}[shift={(0,-11)}]
	
	\def\dotPt{4pt}
	
	\foreach \x/\y in {
	  1/9, 2/10, 7/11, 10/12, 11/13, 12/14, 14/15, 15/16
	}{
	  \draw (\x,0) -- (\y,4);
	}
	\draw[red] (17,0) -- (17,4);
	
	\foreach \x/\y in {
	  1/3, 2/4, 3/5, 4/6, 5/8, 6/9, 7/13
	}{
	  \draw[white, line width=3pt] (\x,4) -- (\y,0);
	  \draw (\x,4) -- (\y,0);
	}
	\draw[white, line width=3pt] (8,4) -- (16,0);
	\draw[blue] (8,4) -- (16,0);
	
	\foreach \x/\y in {
	  1/1, 2/2, 3/3, 4/4, 5/5, 6/6, 7/7, 10/9, 11/10, 12/11, 13/12, 14/13, 15/14, 16/15, 17/16
	}{
	  \draw (\x,8) -- (\y,4);
	}
	\draw[blue] (8,8) -- (8,4);

	\foreach \x/\y in {
	  9/17
	}{
	  \draw[white, line width=3pt] (\x,8) -- (\y,4);
	  \draw[red] (\x,8) -- (\y,4);
	}

	\foreach \y in {0,4,8}{
	  \draw (0,\y) -- (18,\y);
	}
	
	\foreach \x in {1,...,17}{
	  \foreach \y in {0,4,8}{
	    \fill (\x,\y) circle (\dotPt);
	  }
	}
	\node at (20,4) {$=$};
	\end{scope}
	
	\begin{scope}[shift={(22,-7)}]
	
	\def\dotPt{4pt}
	
	\foreach \x/\y in {
	  1/8, 2/9, 7/10, 10/11, 11/12, 12/13, 14/14, 15/15
	}{
	  \draw (\x,-4) -- (\y,0);
	}
	\draw[red] (17,-4) -- (17,0);
	\draw[blue] (16,-4) -- (16,0);
	
	\foreach \x/\y in {
	  1/3, 2/4, 3/5, 4/6, 5/8, 6/9, 7/13
	}{
	  \draw[white, line width=3pt] (\x,0) -- (\y,-4);
	  \draw (\x,0) -- (\y,-4);
	}
	
	\foreach \x/\y in {
	  1/1, 2/2, 3/3, 4/4, 5/5, 6/6, 7/7, 9/8, 10/9, 11/10, 12/11, 13/12, 14/13, 15/14, 16/15
	}{
	  \draw (\x,4) -- (\y,0);
	}
	\draw[red] (17,4) -- (17,0);
	
	\foreach \x/\y in {
	  8/16
	}{
	  \draw[white, line width=3pt] (\x,4) -- (\y,0);
	  \draw[blue] (\x,4) -- (\y,0);
	}
	
	\foreach \x/\y in {
	  1/1, 2/2, 3/3, 4/4, 5/5, 6/6, 7/7, 10/9, 11/10, 12/11, 13/12, 14/13, 15/14, 16/15, 17/16
	}{
	  \draw (\x,8) -- (\y,4);
	}
	\draw[blue] (8,8) -- (8,4);
	
	\foreach \x/\y in {
	  9/17
	}{
	  \draw[white, line width=3pt] (\x,8) -- (\y,4);
	  \draw[red] (\x,8) -- (\y,4);
	}
	
	\foreach \y in {-4, 0,4,8}{
	  \draw (0,\y) -- (18,\y);
	}
	
	\foreach \x in {1,...,17}{
	  \foreach \y in {-4,0,4,8}{
	    \fill (\x,\y) circle (\dotPt);
	  }
	}
	\end{scope}

	\end{tikzpicture}
		\caption{Decompositions of a Lorenz braid.}
		\label{}
	\end{figure}
	By the induction hypothesis, in both cases, this can be verified to coincide with the right-hand side of the theorem.
\end{proof}

Together with \cref{cor:component-number}, we obtain the next result.

\begin{corollary}\label{cor:perm-mat}
	Under the above notation, the number of components of the Lorenz link $\widehat{\mathbb{L}}(p_1, q_1; \dots; p_r, q_r)$ is equal to the multiplicity of the eigenvalue $1$ of the matrix 
	\[
		\lim_{t \to 1} \beta_{P_r + Q_r}(\sigma) = 
		\left(\begin{array}{cc|cc|cc|cc}
		 & I_{q_1} & & & &  & &  \\
		 &  & & I_{q_2} & & & & \\
		 &  &  &  & \cdots & \cdots &  & \\
		 &  & &  & & & & I_{q_r} \\ \hline
		I_{p_1} &  &  &  & & &  & \\ 
		& & I_{p_2} &  & & &  & \\
		&  & & & \cdots & \cdots & & \\
		&  &  & & & & I_{p_r} & 
		\end{array} \right).
	\]
\end{corollary}

Moreover, from the explicit formula, the degree of the Alexander polynomial can also be determined as follows.

\begin{corollary}\label{cor:deg-Alexander}
	Let $\sigma = \mathbb{L}(p_1, q_1; \dots; p_r, q_r)$ be a Lorenz braid whose closure is a knot. Then, the Alexander polynomial $\Delta_{\widehat{\sigma}}(t)$ belongs to $\Z[t]$ and is a monic reciprocal polynomial of degree
	\[
		\deg \Delta_{\widehat{\sigma}} = \sum_{j=1}^r P_j q_j - (P_r + Q_r) + 1.
	\]
\end{corollary}

\begin{proof}
	The absence of $t^{-1}$-terms in $\Delta_{\widehat{\sigma}}(t)$, as well as the fact that its constant term is $1$, follow from the observation that the matrix $s I_{P_r+Q_r} - \beta_{P_r + Q_r}(\sigma)$ at $t=0$ is lower triangular with diagonal entries $s-1, s, \dots, s$. 
	
	By \cref{lem:sym-Alexander}, we have
	\[
		\Delta_{\widehat{\sigma}}(1/t) = (-t)^{-|\sigma| + n -1} \Delta_{\widehat{\sigma}}(t),
	\]
	where $n = P_r + Q_r$. Since $\sigma$ is a positive braid, $|\sigma|$ equals the number of crossings of the braid, namely, $\sum_{j=1}^r P_j q_j$. As $\widehat{\sigma}$ is a knot, the permutation $\pi(\sigma)$ is cyclic. Both $(-1)^{|\sigma|}$ and $(-1)^{n-1}$ give the sign of $\pi(\sigma) \in \mathfrak{S}_n$, hence $|\sigma| - n + 1$ is even. Therefore, we see that $\Delta_{\widehat{\sigma}}(t)$ is a monic reciprocal polynomial of degree $|\sigma| - n + 1$.
\end{proof}

In fact, the above corollary can also be derived from the observation that a Lorenz braid is a positive braid, and the closure of any positive braid yields a fibered knot. For a fibered knot, the degree of the Alexander polynomial equals twice the \emph{genus} of its Seifert surface, and this genus can be computed explicitly from the number of strands and crossings of the corresponding positive braid. See Stallings~\cite{Stallings1978}.

In addition, since Lorenz knots are positive braid knots, the degree of the Alexander polynomial equals not only twice the genus but also twice the \emph{unknotting number} of the knot. See \cite{Rudolph1983, KLMMMS2024}. Hence, it is interesting to ask whether this degree can be expressed in terms of the coefficients of the continued fraction expansion of the corresponding real quadratic irrational for a modular knot.

\begin{remark}
	For $T$-braids $\mathbb{T}((r_1, s_1), (r_2, s_2))$, the third author and Adnan~\cite{ParkAdnan2025} computed its Alexander polynomial under certain conditions, and Adnan, Paiva, and the third author~\cite{AdnanPaivaPark2025} studied the number of components of $T$-link. This corresponds to the case $r=2$ in Lorenz braids. Although their analysis focuses on the double twist case, they also considered the general situation with $s_2 < 0$.
\end{remark}

\section{Which polynomials arise as Alexander polynomials of modular knots}\label{sec:Alexander-finite}

In this section, we consider the following subset of polynomials:
\[
	A_n \coloneqq \{\Delta_{K}(t) : \deg \Delta_K = n, K: \text{modualr knot} \} \subset \Z[t].
\]
As is well known, the Alexander polynomial is a reciprocal (Laurent) polynomial that takes the value $1$ at $t=1$, so it immediately follows that $A_n = \emptyset$ for every odd $n$.

As shown in \cref{cor:deg-Alexander}, if a modular braid $\mathbb{M}$ corresponds to a Lorenz braid $\mathbb{L}(p_1, q_1; \dots; p_r, q_r)$, then the degree of its Alexander polynomial is given by
\begin{align*}
	\deg \Delta_{\widehat{\mathbb{M}}} &= \sum_{j=1}^{r-1} (P_j - 1)q_j + (P_r-1)(q_r-1).
\end{align*}
By \cref{lem:Markov}, Markov moves allow us to assume $p_1 > 1$ and $q_r > 1$ without affecting the Alexander polynomial. Therefore, the order of $A_n$ can be bounded above by
\begin{align}\label{eq:count}
	&\#\left\{(p_1, q_1; \dots; p_r, q_r) \in \Z_{\ge 1}^{2r} : r \ge 1, p_1, q_r > 1, \sum_{j=1}^{r-1} (P_j-1)q_j + (P_r-1)(q_r-1) = n \right\}\\
	&= \sum_{r=1}^\infty \sum_{\substack{0 < m_1 < m_2 < \cdots < m_r \\ n = m_1 d_1 + m_2 d_2 + \cdots + m_r d_r}} 1. \nonumber
\end{align}
This number coincides with the number of partitions of $n$, denoted by $p(n)$, and thus we obtain the following trivial upper bound, (see~\cite{AndrewsEriksson2004} on partition theory).

\begin{theorem}\label{thm:fin-Alexander}
	For a given even degree $n$, we have $\# A_n \le p(n)$. In particular, $A_n$ is a finite set.
\end{theorem}

As the examples below illustrate, the order of $A_n$ is much smaller than $p(n)$, since the closure of a corresponding Lorenz braid $\mathbb{L}(p_1, q_1; \dots; p_r, q_r)$ may fail to be a knot, and different tuples may define the same knot or the same Alexander polynomial.

\begin{example}
	For $n=2$, the only solutions $(p_1, q_1; \dots; p_r, q_r) \in \Z_{\ge 1}^{2r}$ with $p_1, q_r > 1$ to
	\[
		\sum_{j=1}^{r-1} (P_j-1)q_j + (P_r-1)(q_r-1) = 2
	\]
	are $(p_1, q_1) = (2,3)$ and $(3,2)$. The closures of the modular (Lorenz) braids $\mathbb{L}(2,3)$ and $\mathbb{L}(3,2)$ both yield knots, and in both cases the Alexander polynomial is $t^2 - t + 1$. Therefore, $A_2 = \{t^2 - t + 1\}$ and $\#A_2 = 1 < p(2) = 2$. 
	\begin{figure}[H]
	\centering
	\begin{tikzpicture}[line width=1pt, scale=0.4]
	\begin{scope}[shift={(0,0)}]
		\draw (0,0) -- (10,0);
		\draw (0,6) -- (10,6);
		
		\def \dotPt {4pt}; 
		\foreach \x in {1,3,5,7,9}{\fill (\x,0) circle (\dotPt);}
		\foreach \x in {1,3,5,7,9}{\fill (\x,6) circle (\dotPt);}
	
		\foreach \a/\b in {1/7, 3/9}{\draw (\a,0) -- (\b,6);}
		
		\foreach \a/\b in {5/1, 7/3, 9/5}{
		  \draw[white, line width=4pt] (\a-1/4,3/8) -- (\b+1/4,45/8);
		  \draw (\a,0) -- (\b,6);
		}
		
		\foreach \x in {1,2,3,4,5}{\node at (2*\x-1,6.5) {\tiny $\x$};}
		\foreach \x in {1,2,3,4,5}{\node at (2*\x-1,-0.5) {\tiny $\x$};}
		
		\draw[red] (21/5,6/5) circle[radius=1/2];
		\draw[red] (27/5,12/5) circle[radius=1/2];
	
		\node at (5,-2) {$LLRLR$};	
		\node at (5,-3.3) {$3x^2 - 6xy - 5y^2$};	
	\end{scope}
	
	\begin{scope}[shift={(14,0)}]
		\draw (0,0) -- (10,0);
		\draw (0,6) -- (10,6);
		
		\def \dotPt {4pt}; 
		\foreach \x in {1,3,5,7,9}{\fill (\x,0) circle (\dotPt);}
		\foreach \x in {1,3,5,7,9}{\fill (\x,6) circle (\dotPt);}
	
		\foreach \a/\b in {1/5, 3/7, 5/9}{\draw (\a,0) -- (\b,6);}
		
		\foreach \a/\b in {7/1, 9/3}{
		  \draw[white, line width=4pt] (\a-1/4,1/4) -- (\b+1/4,23/4);
		  \draw (\a,0) -- (\b,6);
		}
		
		\foreach \x in {1,2,3,4,5}{\node at (2*\x-1,6.5) {\tiny $\x$};}
		\foreach \x in {1,2,3,4,5}{\node at (2*\x-1,-0.5) {\tiny $\x$};}
		
		\draw[red] (23/5,12/5) circle[radius=1/2];
		\draw[red] (29/5,6/5) circle[radius=1/2];
	
		\node at (5,-2) {$LRLRR$};
		\node at (5,-3.3) {$5x^2 - 6xy - 3y^2$};	
	\end{scope}
	
	\end{tikzpicture}
		\caption{Two braids corresponding to the partitions $1+1$ and $2$ of $2$.}
		\label{n=2}
	\end{figure}
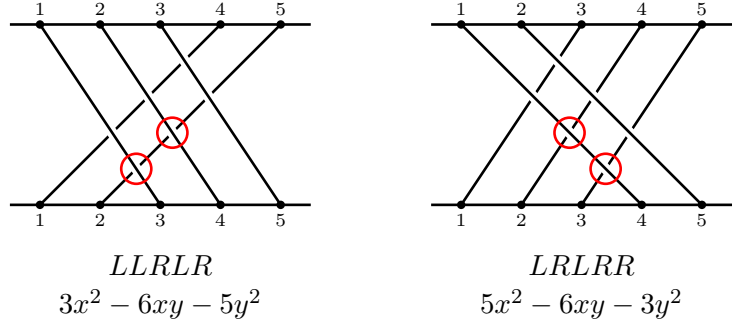
\end{example}

For example, Ghys--Leys~\cite{GhysLeys} noted that the figure-eight knot cannot be realized as a modular (Lorenz) knot, and this follows from the fact that its Alexander polynomial is $-t + 3 - t^{-1}$ (up to a unit in $\Z[t,t^{-1}]$). See also the list by Birman--Kofman~\cite[Table 1]{BirmanKofman2009} of hyperbolic knots that can be realized as Lorenz knots.

\begin{example}
	For $n=4$, the tuples $(p_1, q_1; \dots)$ satisfying the above conditions are $(2,5)$, $(3,3)$, $(5,2)$, $(2,1; 2,2)$, and $(2,2; 1,2)$. Among these, the case $(3,3)$ does not define a knot, while for the other four cases the Alexander polynomial is $t^4 - t^3 + t^2 - t + 1$. More precisely, we see that all of these closures determine the same knot. Therefore,
	\[
		A_4 = \{t^4 - t^3 + t^2 - t + 1\}.
	\]
	\begin{figure}[H]
	\centering
	\begin{tikzpicture}[line width=1pt, scale=0.5]
	\begin{scope}[shift={(0,0)}]
	\def\dotPt{4pt}
	
	\foreach \x/\y in {
	  1/3, 2/4, 3/5, 4/6, 5/7
	}{
	  \draw (\x,0) -- (\y,4);
	}
	
	\foreach \x/\y in {
	  1/6, 2/7
	}{
	  \draw[white, line width=3pt] (\x,4) -- (\y,0);
	  \draw (\x,4) -- (\y,0);
	}
	
	\foreach \y in {0,4}{
	  \draw (0,\y) -- (8,\y);
	}
	
	\foreach \x in {1,...,7}{
	  \foreach \y in {0,4}{
	    \fill (\x,\y) circle (\dotPt);
	  }
	}
	
	\foreach \i in {1,...,7}{
	  \node at (\i,4.5) {\tiny $\i$};
	  \node at (\i,-0.5) {\tiny $\i$};
	}
	\draw[red] (22/7,16/7) circle[radius=1/3];
	\draw[red] (27/7,12/7) circle[radius=1/3];
	\draw[red] (32/7,8/7) circle[radius=1/3];
	\draw[red] (37/7,4/7) circle[radius=1/3];
	
	\node at (4,-1.5) {$LRRLRRR$};	
	\node at (4,-2.5) {$11x^2 - 12xy - 4y^2$};	
	
	\end{scope}
	
	\begin{scope}[shift={(11,0)}]
	\def\dotPt{4pt}
	
	\foreach \x/\y in {
	  1/4, 2/5, 4/6, 5/7
	}{
	  \draw (\x,0) -- (\y,4);
	}
	
	\foreach \x/\y in {
	  1/3, 2/6, 3/7
	}{
	  \draw[white, line width=3pt] (\x,4) -- (\y,0);
	  \draw (\x,4) -- (\y,0);
	}
	
	\foreach \y in {0,4}{
	  \draw (0,\y) -- (8,\y);
	}
	
	\foreach \x in {1,...,7}{
	  \foreach \y in {0,4}{
	    \fill (\x,\y) circle (\dotPt);
	  }
	}
	
	\foreach \i in {1,...,7}{
	  \node at (\i,4.5) {\tiny $\i$};
	  \node at (\i,-0.5) {\tiny $\i$};
	}
	
	\draw[red] (13/5,4/5) circle[radius=1/3];
	\draw[red] (26/7,16/7) circle[radius=1/3];
	\draw[red] (14/3,4/3) circle[radius=1/3];
	\draw[red] (16/3,2/3) circle[radius=1/3];
	
	\node at (4,-1.5) {$LLRRLRR$};	
	\node at (4,-2.5) {$8x^2 - 16xy - 7y^2$};
	
	\end{scope}
	\end{tikzpicture}
	\end{figure}
	
	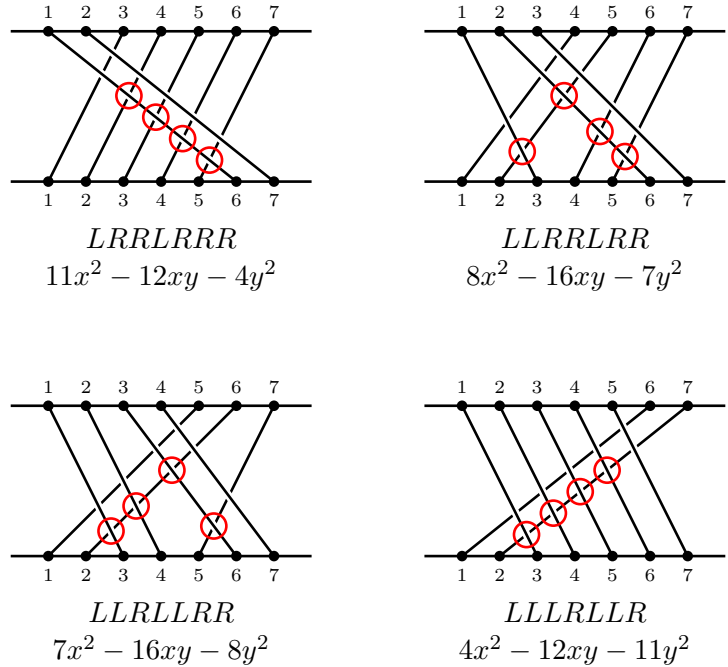
\begin{figure}[H]
	\centering
	\begin{tikzpicture}[line width=1pt, scale=0.5]
	\begin{scope}[shift={(0,0)}]
	\def\dotPt{4pt}
	
	\foreach \x/\y in {
	  1/5, 2/6, 5/7
	}{
	  \draw (\x,0) -- (\y,4);
	}
	
	\foreach \x/\y in {
	  1/3, 2/4, 3/6, 4/7
	}{
	  \draw[white, line width=3pt] (\x,4) -- (\y,0);
	  \draw (\x,4) -- (\y,0);
	}
	
	\foreach \y in {0,4}{
	  \draw (0,\y) -- (8,\y);
	}
	
	\foreach \x in {1,...,7}{
	  \foreach \y in {0,4}{
	    \fill (\x,\y) circle (\dotPt);
	  }
	}
	
	\foreach \i in {1,...,7}{
	  \node at (\i,4.5) {\tiny $\i$};
	  \node at (\i,-0.5) {\tiny $\i$};
	}
	\draw[red] (8-13/5,4/5) circle[radius=1/3];
	\draw[red] (8-26/7,16/7) circle[radius=1/3];
	\draw[red] (8-14/3,4/3) circle[radius=1/3];
	\draw[red] (8-16/3,2/3) circle[radius=1/3];
	
	\node at (4,-1.5) {$LLRLLRR$};	
	\node at (4,-2.5) {$7x^2 - 16xy - 8y^2$};
	
	\end{scope}
	
	\begin{scope}[shift={(11,0)}]
	\def\dotPt{4pt}
	
	\foreach \x/\y in {
	  1/6, 2/7
	}{
	  \draw (\x,0) -- (\y,4);
	}
	
	\foreach \x/\y in {
	  1/3, 2/4, 3/5, 4/6, 5/7
	}{
	  \draw[white, line width=3pt] (\x,4) -- (\y,0);
	  \draw (\x,4) -- (\y,0);
	}
	
	\foreach \y in {0,4}{
	  \draw (0,\y) -- (8,\y);
	}
	
	\foreach \x in {1,...,7}{
	  \foreach \y in {0,4}{
	    \fill (\x,\y) circle (\dotPt);
	  }
	}
	
	\foreach \i in {1,...,7}{
	  \node at (\i,4.5) {\tiny $\i$};
	  \node at (\i,-0.5) {\tiny $\i$};
	}
	
	\draw[red] (8-22/7,16/7) circle[radius=1/3];
	\draw[red] (8-27/7,12/7) circle[radius=1/3];
	\draw[red] (8-32/7,8/7) circle[radius=1/3];
	\draw[red] (8-37/7,4/7) circle[radius=1/3];
	
	\node at (4,-1.5) {$LLLRLLR$};	
	\node at (4,-2.5) {$4x^2 - 12xy - 11y^2$};
	
	\end{scope}
	
	\end{tikzpicture}
		\caption{Four braids corresponding to the partitions $4$, $3+1$, $2+1+1$, and $1+1+1+1$ of $4$. The braid corresponding to the partition $2+2$ does not define a knot.}
		\label{n=4}
	\end{figure}
\end{example}

\begin{example}
	For $n=6$, the tuples $(p_1, q_1; \dots)$ satisfying the conditions are $(2,7)$, $(3,4)$, $(4,3)$, $(7,2)$, $(2,2; 1,3)$, $(2,4; 1,2)$, $(2,3; 2,2)$, $(2,2; 3,2)$, $(2,1; 4,2)$, $(3,1; 2,2)$, and $(2,1; 1,1; 1,2)$. Among these, the cases $(2,1; 2,3)$ and $(3,2; 1,2)$ do not define a knot. The resulting set is
	\[
		A_6 = \{t^6 - t^5 + t^4 - t^3 +t^2 -t+1, t^6-t^5+t^3-t+1\}.
	\]
\end{example}

For each $n$, let $k(n)$ denote the number of tuples counted by \eqref{eq:count} whose corresponding Lorenz link is a knot. In the examples above, we have $k(2) = 2$, $k(4) = 4$, and $k(6) = 9$. Performing similar computations for other $n$ yields the following list.

\begin{center}
\begin{tabular}{c|ccccccccccccccccccccc}
$n$   & 2 & 4 & 6 & 8 & 10 & 12 & 14 & 16 & 18 & 20 & 22 & 24 & 26 & 28 & 30 \\
\hline \hline
$\# A_n$ & 1 & 1 & 2 & 2 & 3 & 5 & 6 & 9 & 14 & 20 & 27 & 41 & 55 & 80 & 113  \\ \hline
$k(n)$ & 2 & 4 & 9 & 16 & 28 & 52 & 84 & 134 & 223 & 349 & 532 & 824 & 1226 & 1808 & 2693 \\ \hline
$p(n)$ & 2 & 5 & 11 & 22 & 42 & 77 & 135 & 231 & 385 & 627 & 1002 & 1575 & 2436 & 3718 & 5604 
\end{tabular}
\end{center}

As suggested by this table, $p(n)$ serves only as a trivial upper bound, and there exist many tuples $(p_1, q_1; ...)$ whose Lorenz braids determine the same knot or share the same Alexander polynomial. For instance, it is immediate that taking the conjugate of a partition corresponds, in terms of the Lorenz braid, to flipping the braid over like turning a sheet of paper, or in terms of the word, to swapping $L$ and $R$. The knot obtained by this operation remains unchanged. Moreover, as Dehornoy~\cite[Proposition 4.20]{Dehornoy2011} showed, there is another nontrivial duality, in which the word is read in reverse order while swapping $L$ and $R$. For example, $W = L^4 R^3 L^4 R^3 LR^2$ is transformed into $W' = L^2 RL^3 R^4 L^3 R^4$. The corresponding modular braid is given below. Although its relationship with the modular braid $\mathbb{M}(W) = \mathbb{L}(2,4; 1,2; 3,1; 2,2)$ shown in \cref{Lorenz-braid-example} is not immediately clear from the viewpoint of braid diagrams, the knot obtained by this operation also remains unchanged.

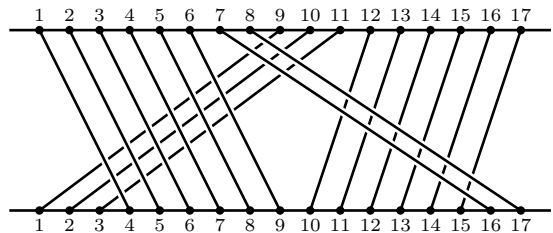
\begin{figure}[H]
	\centering
	\begin{tikzpicture}[line width=1pt, scale=0.4]
	\begin{scope}[shift={(0,0)}]
	\def\dotPt{4pt}
	
	\foreach \x/\y in {
	  1/9, 2/10, 3/11, 10/12, 11/13, 12/14, 13/15, 14/16, 15/17 
	}{
	  \draw (\x,0) -- (\y,6);
	}
	
	\foreach \x/\y in {
	  1/4, 2/5, 3/6, 4/7, 5/8, 6/9, 7/16, 8/17
	}{
	  \draw[white, line width=3pt] (\x,6) -- (\y,0);
	  \draw (\x,6) -- (\y,0);
	}
	
	\foreach \y in {0,6}{
	  \draw (0,\y) -- (18,\y);
	}
	
	\foreach \x in {1,...,17}{
	  \foreach \y in {0,6}{
	    \fill (\x,\y) circle (\dotPt);
	  }
	}
	
	\foreach \x in {1,...,17}{\node at (\x,6.5) {\tiny $\x$};}
	\foreach \x in {1,...,17}{\node at (\x,-0.5) {\tiny $\x$};}
	\end{scope}
	\end{tikzpicture}
	\caption{Modular braid $\mathbb{M}(W') = \mathbb{L}(3,6; 6,2)$.}
\end{figure}
To clarify the mechanism, the real quadratic irrationals associated with $W$ and $W'$ are
\begin{align*}
	w &= [\overline{4,3,4,3,1,2}] = \frac{150 + \sqrt{31506}}{76},\\
	w' &= [\overline{2,1,3,4,3,4}] \sim [-1, 1, 1, 1, \overline{3,4,3,4,2,1}] = \frac{150 - \sqrt{31506}}{76},
\end{align*}
respectively. In general, this operation produces the Galois conjugate. According to the definition of modular knots via the geodesic flow, it simply reverses the orientation of the modular knot.

However, the table above suggests that there are many other nontrivial coincidences among modular knots beyond these examples. It would be interesting to know whether one can provide a criterion for when they define the same knot, or a sharper upper bound.

\section{Coefficients of the Alexander polynomials of modular knots}\label{sec:Coeff-Alexander}

As an analogue of Schur and Suzuki's theorem~\cite{Suzuki1987} for cyclotomic polynomials, we obtain the following result.

\begin{theorem}\label{thm:coeff-Alexander}
	Every integer appears as a coefficient of the Alexander polynomial of a modular knot.
\end{theorem}

\begin{proof}
	It suffices to explicitly construct a modular knot whose Alexander polynomial has each given integer as one of its coefficients. For a positive even integer $n$, we consider the word $W_n = L^{m+1} R^{m+1} L R^{m+1} LRLR L^m R$ with $m = n/2+1$. The corresponding modular braid is
	\[
		\mathbb{M}(W_n) = \mathbb{L}(2,n; 3,4; n+2,2).
	\]
	Using the template deformation given by Birman--Williams~\cite[Figure 5.1]{BirmanWilliams1983}, we construct a braid with the same closure as $\mathbb{M}(W_n)$ that realizes the braid index of this knot (which, as shown in~\cref{sec:braid-index}, equals the number of subword $LR$, equal to $5$ in this case). Then compute its Alexander polynomial, (see also Birman--Kofman~\cite[Figures 4 and 5]{BirmanKofman2009}).
	\begin{figure}[H]
	\centering
	\begin{tikzpicture}[line width=1pt, scale=0.3]
	\begin{scope}[shift={(5,-5)}]
	\def\dotPt{4pt}
	
	\foreach \x/\y in {
	  1/9, 2/10, 5/11, 6/12, 7/13
	}{
	  \draw (\x,0) -- (\y,6);
	}
	\foreach \x/\y in {
	  12/14, 13/15, 14/16, 15/17
	}{
	  \draw[orange] (\x,0) -- (\y,6);
	}
	
	\foreach \x/\y in {
	  1/3, 2/4, 3/8
	}{
	  \draw[white, line width=3pt] (\x,6) -- (\y,0);
	  \draw[red] (\x,6) -- (\y,0);
	}
	
	\foreach \x/\y in {
	  4/9, 5/10, 6/11, 7/16, 8/17
	}{
	  \draw[white, line width=3pt] (\x,6) -- (\y,0);
	  \draw[blue] (\x,6) -- (\y,0);
	}
	
	\foreach \x in {
	  1,2,3
	}{
	  \draw[red, line width=0.5pt] (\x,6) -- (\x,6+0.5*\x) -- (0.5-0.5*\x,6+0.5*\x) -- (0.5-0.5*\x,-0.5*\x) -- (\x,-0.5*\x) -- (\x,0);
	}
	\foreach \x in {
	  4,5,6,7,8
	}{
	  \draw[blue, line width=0.5pt] (\x,6) -- (\x,6+0.5*\x) -- (0.5-0.5*\x,6+0.5*\x) -- (0.5-0.5*\x,-0.5*\x) -- (\x,-0.5*\x) -- (\x,0);
	}
	
	\foreach \x in {
	  1,2,3,4
	}{
	  \draw[orange, line width=0.5pt] (18-\x,6) -- (18-\x,6+0.5*\x) -- (18-0.5+0.5*\x,6+0.5*\x) -- (18-0.5+0.5*\x,-0.5*\x) -- (18-\x,-0.5*\x) -- (18-\x,0);
	}
	
	\foreach \x in {
	 9,10,11,12,13
	}{
	 \draw (\x,6) -- (\x,8); \draw (\x,0) -- (\x, -2);
	 \fill (\x, 8) circle (\dotPt); \fill (\x, -2) circle (\dotPt);
	}
	
	\foreach \x in {1,...,17}{
	  \foreach \y in {0,6}{
	    \fill (\x,\y) circle (\dotPt);
	  }
	}
	
	\foreach \i in {1,...,17}{
	  \node at (\i+0.4,6.1) {\scalebox{0.4}{$\i$}};
	  \node at (\i+0.4,-0.1) {\scalebox{0.4}{$\i$}};
	}
	
	\node at (21.5,3) {\Large{\boldsymbol{$\rightarrow$}}};
	
	\end{scope}
	
	\begin{scope}[shift={(25,-10)}]
	\def\dotPt{4pt}

	\foreach \x/\y in {
	  1/9, 2/10, 5/11, 6/12, 7/13
	}{
	  \draw (\x,21) -- (\y,24); \draw (\y,24) -- (\y,25); \draw (\y,-3) -- (\y,-5);
	}
	\foreach \x/\y in {
	  12/14, 13/15, 14/16, 15/17
	}{
	  \draw[orange] (\x,-3) -- (\y,0);
	}
	
	\foreach \x/\y in {
	  1/3, 2/4, 3/8
	}{
	  \draw[white, line width=3pt] (\x,24) -- (\y,21);
	  \draw[red] (\x,24) -- (\y,21);
	}
	
	\draw[blue] (4,21) -- (4,20);
	\draw[blue] (8,16) -- (8,13) -- (7,12) -- (7,10) -- (6,9) -- (6,7) -- (5,6) -- (5,4) -- (4,3) -- (4,0);
	\draw[blue] (5,21) -- (5,20) -- (4,19) -- (4,16);
	\draw[blue] (8,12) -- (8,9) -- (7,8) -- (7,6) -- (6,5) -- (6,3) -- (5,2) -- (5,0);
	\draw[blue] (6,21) -- (6,19) -- (5,18) -- (5,16) -- (4,15) -- (4,12);
	\draw[blue] (8,8) -- (8,5) -- (7,4) -- (7,2) -- (6,1) -- (6,0);
	\draw[blue] (7,21) -- (7,18) -- (6,17) -- (6,15) -- (5,14) -- (5,12) -- (4,11) -- (4,8);
	\draw[blue] (8,4) -- (8,1) -- (7,0);
	\draw[blue] (8,21) -- (8,17) -- (7,16) -- (7,14) -- (6,13) -- (6,11) -- (5,10) -- (5,8) -- (4,7) -- (4,4);
	
	\foreach \x/\y in {
	  4/9, 5/10, 6/11, 7/16, 8/17
	}{
	  \draw[white, line width=3pt] (\x,0) -- (\y,-3);
	  \draw[blue] (\x,0) -- (\y,-3);
	}
	
	\foreach \x/\y in {
	  20/16, 16/12, 12/8, 8/4, 4/0
	}{
	  \draw[white, line width=3pt] (4,\x) -- (8,\y);
	  \draw[blue] (4,\x) -- (8,\y);
	}
	
	\foreach \x in {
	  1,2,3
	}{
	  \draw[red, line width=0.5pt] (\x,24) -- (\x,24+0.5*\x) -- (0.5-0.5*\x,24+0.5*\x) -- (0.5-0.5*\x,21-0.5*\x) -- (\x,21-0.5*\x) -- (\x,21);
	}
	
	\foreach \x in {
	  1,2,3,4
	}{
	  \draw[orange, line width=0.5pt] (18-\x,0) -- (18-\x,0+0.5*\x) -- (18-0.5+0.5*\x,0+0.5*\x) -- (18-0.5+0.5*\x,-3-0.5*\x) -- (18-\x,-3-0.5*\x) -- (18-\x,-3);
	}

	\foreach \x in {1,2,3}{
	  \foreach \y in {24}{
	    \fill (\x,\y) circle (\dotPt);  \node at (\x+0.4,\y+0.1) {\scalebox{0.4}{$\x$}};
	  }
	}
	\foreach \x in {1,2,3}{
	  \foreach \y in {21}{
	    \fill (\x,\y) circle (\dotPt);  \node at (\x+0.4,\y-0.2) {\scalebox{0.4}{$\x$}};
	  }
	}
	\foreach \x in {4,5,6,7,8}{
	  \foreach \y in {21}{
	    \fill (\x,\y) circle (\dotPt); \node at (\x+0.4,\y-0.2) {\scalebox{0.4}{$\x$}};
	  }
	}
	\foreach \x in {4,5,6,7,8}{
	  \foreach \y in {0}{
	    \fill (\x,\y) circle (\dotPt); \node at (\x+0.4,\y+0.1) {\scalebox{0.4}{$\x$}};
	  }
	}
	\foreach \x in {9,10,11,12,13}{
	  \foreach \y in {25}{
	    \fill (\x,\y) circle (\dotPt); \node at (\x+0.4,\y+0.1) {\scalebox{0.4}{$\x$}};
	  }
	}
	\foreach \x in {9,10,11,12,13}{
	  \foreach \y in {-5,-3}{
	    \fill (\x,\y) circle (\dotPt); \node at (\x+0.4,\y-0.2) {\scalebox{0.4}{$\x$}};
	  }
	}
	\foreach \x in {14,15,16,17}{
	  \foreach \y in {-3,0}{
	    \fill (\x,\y) circle (\dotPt); \node at (\x+0.4,\y-0.2) {\scalebox{0.4}{$\x$}};
	  }
	}
	\node at (18,13) {\Large{\boldsymbol{$\rightarrow$}}};

	\end{scope}
	
	\begin{scope}[shift={(44,-10)}]
	\def\dotPt{4pt}
	
	\draw (4,29) -- (4,28); 
	\draw (5,29) -- (5,28);
	\draw (6,29) -- (6,25) -- (5,24) -- (5,21);
	\draw (7,29) -- (7,24) -- (6,23) -- (6,21);
	\draw (8,29) -- (8,23) -- (7,22) -- (7,21);
	\foreach \x in {4,5,6,7,8}{
	\draw (\x,-6) -- (\x,-5);}
	
	\draw[red] (5,28) -- (4,27);
	\draw[red] (5,27) -- (4,26);
	\draw[red] (5,26) -- (4,25) -- (4,21);
	\draw[red] (8,22) -- (8,21);
	
	\draw[blue] (4,21) -- (4,20);
	\draw[blue] (8,16) -- (8,13) -- (7,12) -- (7,10) -- (6,9) -- (6,7) -- (5,6) -- (5,4) -- (4,3) -- (4,-5);
	\draw[blue] (5,21) -- (5,20) -- (4,19) -- (4,16);
	\draw[blue] (8,12) -- (8,9) -- (7,8) -- (7,6) -- (6,5) -- (6,3) -- (5,2) -- (5,-5);
	\draw[blue] (6,21) -- (6,19) -- (5,18) -- (5,16) -- (4,15) -- (4,12);
	\draw[blue] (8,8) -- (8,5) -- (7,4) -- (7,2) -- (6,1) -- (6,-5);
	\draw[blue] (7,21) -- (7,18) -- (6,17) -- (6,15) -- (5,14) -- (5,12) -- (4,11) -- (4,8);
	\draw[blue] (8,4) -- (8,1) -- (7,0);
	\draw[blue] (8,21) -- (8,17) -- (7,16) -- (7,14) -- (6,13) -- (6,11) -- (5,10) -- (5,8) -- (4,7) -- (4,4);

	\foreach \x/\y in {
	  20/16, 16/12, 12/8, 8/4, 4/0
	}{
	  \draw[white, line width=3pt] (4,\x) -- (8,\y);
	  \draw[blue] (4,\x) -- (8,\y);
	}
	
	\foreach \x/\y/\xx/\yy in {
	  4/28/5/27, 4/27/5/26, 4/26/8/22
	}{
	  \draw[white, line width=3pt] (\x+0.1,\y-0.1) -- (\xx-0.1,\yy+0.1);
	  \draw[red] (\x,\y) -- (\xx,\yy);
	}
	
	\draw[orange] (7,0) -- (7,-1);
	\draw[orange] (8,0) -- (8,-1) -- (7,-2);
	\draw[orange] (8,-2) -- (7,-3);
	\draw[orange] (8,-3) -- (7,-4);
	\draw[orange] (8,-4) -- (7,-5);
	\foreach \y/\yy in {
	  -1/-2, -2/-3, -3/-4, -4/-5
	}{
	  \draw[white, line width=3pt] (7+0.1,\y-0.1) -- (8-0.1,\yy+0.1);
	  \draw[orange] (7,\y) -- (8,\yy);
	}
	
	\foreach \x in {4,5,6,7,8}{
	  \foreach \y in {-6,0,21,29}{
	    \fill (\x,\y) circle (\dotPt); 
	  }
	}
	
	\end{scope}

	\end{tikzpicture}
		\caption{The braid $\sigma_1^n (\sigma_1 \sigma_2 \sigma_3 \sigma_4)^6 \sigma_4^{n+2} \in B_5$ has the same closure as the above modular braid $\mathbb{M}(W_n)$. This figure illustrates the case $n=2$.}
		\label{}
	\end{figure}
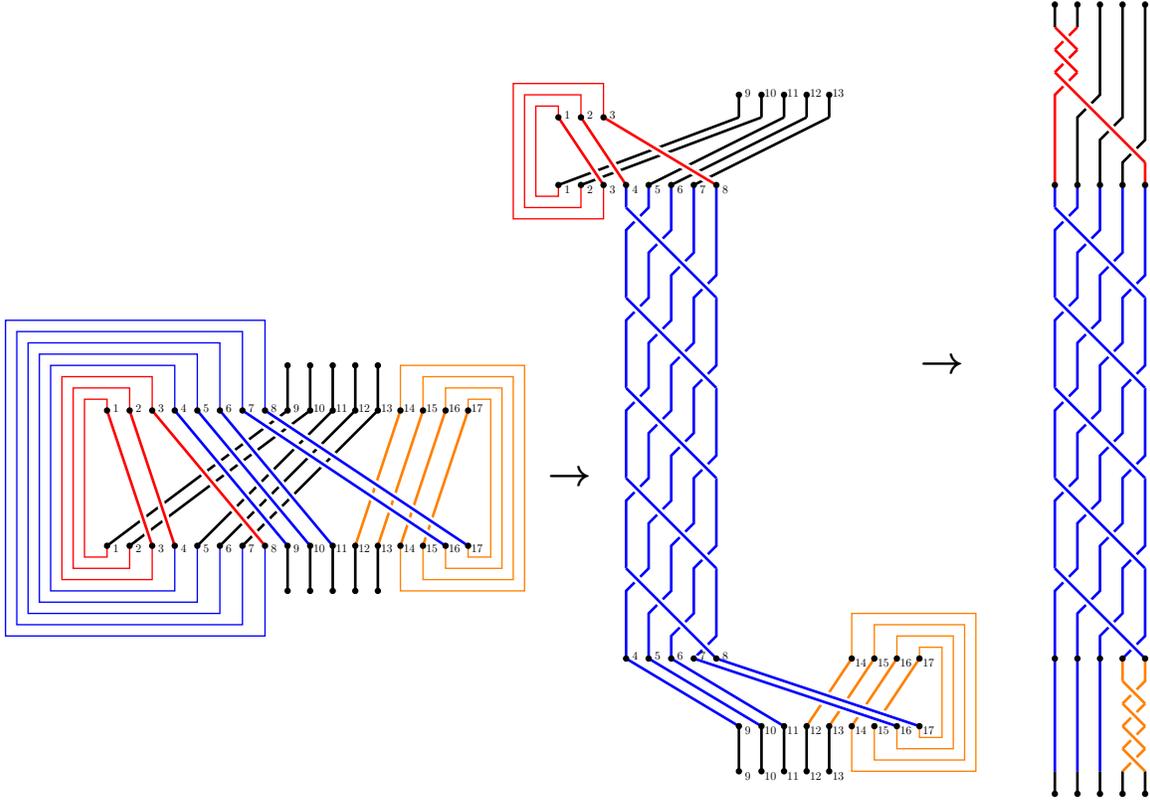
	Hence, by using the expression in \eqref{def:Alexander-second}, we obtain
	\[
		\Delta_{\widehat{\mathbb{M}}(W_n)}(t) = \frac{1}{1+t+t^2+t^3+t^4} \det(I_4 - \beta_5^r(\sigma_1^n (\sigma_1 \sigma_2 \sigma_3 \sigma_4)^6 \sigma_4^{n+2})).
	\]
	Since
	\[
		\beta_5^r(\sigma_1^n (\sigma_1 \sigma_2 \sigma_3 \sigma_4)^6 \sigma_4^{n+2}) = \pmat{
 -t^{n+6} & -t^{n+7} & -t^{n+8} & \frac{-t^{n+9} (1+t^{n+3})}{1+t} \\
 \frac{t^5(1+t^{n+1})}{1+t} & \frac{-t^7 (1-t^n)}{1+t} & \frac{-t^8 (1-t^n)}{1+t} & \frac{-t^9 (1-t^n) (1+t^{n+3})}{(1+t)^2} \\
 0 & t^5 & 0 & 0 \\
 0 & 0 & t^5 & \frac{t^6(1-t^{n+2})}{1+t}},
	\]
	(which can be checked by induction), we have
	\begin{align*}
		\Delta_{\widehat{\mathbb{M}}(W_n)}(t) &= \frac{1}{(1+t)^2} (1+t-t^2 -t^3 +t^5 -t^9 +t^{10}\\
			&\qquad +t^{n+6} -t^{n+10} + t^{n+11} +t^{n+13} -t^{n+14} + t^{n+18} \\
			&\qquad +t^{2n+14}  -t^{2n+15} +t^{2n+19} -t^{2n+21} -t^{2n+22} +t^{2n+23} +t^{2n+24}).
	\end{align*}
	This defines a polynomial, which for $n \ge 10$ has the following expression (the cases $2 \le n \le 8$ are not written explicitly, but they are also polynomials).
	\begin{align*}
			\Delta_{\widehat{\mathbb{M}}(W_n)}(t) &= 1-t+t^5-2t^6+3t^7-4t^8+4t^9-3t^{10}+2t^{11}-t^{12}\\
			&\qquad +t^{14} - 2t^{15} + 3t^{16} - \cdots - (n-8)t^{n+5}\\
			&\qquad + (n-6)t^{n+6} - (n-4)t^{n+7} + (n-2)t^{n+8} - n t^{n+9}\\
			&\qquad + (n+1)t^{n+10} - (n+1)t^{n+11} + (n+1)t^{n+12}\\
			&\qquad -n t^{n+13} + (n-2)t^{n+14} - (n-4)t^{n+15} + (n-6)t^{n+16}\\
			&\qquad - (n-8)t^{n+17} + \cdots + 3t^{2n+6} - 2t^{2n+7} + t^{2n+8}\\
			&\qquad -t^{2n+10} + 2t^{2n+11} - 3t^{2n+12} + 4t^{2n+13} - 4t^{2n+14} \\
			&\qquad \qquad + 3t^{2n+15} -2t^{2n+16} + t^{2n+17} - t^{2n+21} + t^{2n+22}.
	\end{align*}
	Therefore, negative even and positive odd integers appear as coefficients for sufficiently large $n$, while a negative odd integer $k \le -3$ appears when $n=-k-1$, and a positive even integer $k$ appears when $n=k+2$.
\end{proof}

The Alexander polynomials constructed in the proof above exhibit alternating signs among their nonzero coefficients. 
Such alternation of signs also occurs for prime alternating knots~\cite[Theorem 2.13]{Crowell1959} and for $L$-space knots~\cite[Corollary 1.3]{OzsvathSzabo2005}, which include all torus knots.
In contrast, for modular knots, as shown in the following theorem, there exist sequences of coefficients with the same sign of arbitrarily long length.

\begin{theorem}\label{thm:neg-long}
	For the coefficients of the Alexander polynomial of a modular knot, negative coefficients can appear consecutively for an arbitrarily long run.
\end{theorem}

\begin{proof}
	For a positive integer $n$, we consider the word $W_n = (LLR)^n (LR)^n R$. The corresponding modular braid is
	\[
		\mathbb{M}(W_n) = \mathbb{L}(n,n; n+1, 2n).
	\]
	Through the identification with the $T$-braid described in \cref{rem:T-braid}, its closure coincides with $\widehat{\mathbb{T}}((n,n), (2n+1, 2n))$ and with the twisted torus knot $T(2n+1, 2n; n, 1)$ in the notation of Park--Adnan~\cite{ParkAdnan2025}. Although they use the knot group and Fox derivative calculus to obtain the formula for the Alexander polynomial of twisted torus knots, as noted in \cref{rem:Alex-def}, our definition using the Burau representation agrees with theirs up to multiplication by a unit. Applying \cite[Theorem 1]{ParkAdnan2025}, its Alexander polynomial is given as follows.
	\[
		\Delta_{\widehat{\mathbb{M}}(W_n)}(t) = \frac{1-t}{1-t^n} \frac{-t^{-n(2n+1)}}{(1-t^{2n+1})(1-t^{2n})} \bigg(\widetilde{X}(t) Y(t) - X(t) \widetilde{Y}(t) \bigg),
	\]
	where
	\begin{align*}
		X(t) &= 1 - t^{n(5n+1)} - (1-t^n) \sum_{j=1}^{n-1} t^{2n^2 + (3n+1)j},\\
		\widetilde{X}(t) &= 1-t^{n(2n+1)},\\
		Y(t) &= 1 - t^{n(5n+1)} - (1-t^n) \sum_{j=1}^{n-1} t^{2n^2 + (3j+1)n},\\
		\widetilde{Y}(t) &= 1-t^{2n(n+1)}.
	\end{align*}
	
	In the following, we assume that $4 \le n \equiv 1 \pmod{3}$ and consider the $(2n^2+k)$-th coefficient of $\Delta_{\widehat{\mathbb{M}}(W_n)}(t)$ for $1 \le k \le (2n+1)/3$. First, by direct calculation, we obtain
	\begin{align*}
		\Delta_{\widehat{\mathbb{M}}(W_n)}(t) = \frac{1-t}{(1-t^{2n+1})(1-t^{2n})} \bigg(1-t^{n(5n+1)} + (1-t^{n(2n+1)}) \sum_{j=1}^{n-1} t^{3nj} - (1-t^{2n(n+1)})\sum_{j=1}^{n-1} t^{(3n+1)j-n} \bigg).
	\end{align*}
	Since
	\[
		\frac{1-t}{(1-t^{2n+1})(1-t^{2n})} = \sum_{j=0}^\infty (t^{2nj} - t^{(2n+1)j+1}),
	\]
	we have
	\[
		\Delta_{\widehat{\mathbb{M}}(W_n)}(t) \equiv \left(\sum_{j=0}^n t^{2nj} - \sum_{j=0}^{n-1} t^{(2n+1)j+1} \right) \left(1 + \sum_{j=1}^{n-1} (t^{3nj} - t^{(3n+1)j-n}) \right) \pmod{t^{2n^2+n}}.
	\]
	For each $1 \le k \le (2n+1)/3$, the term $t^{2n^2+k}$ can appears only when one or more of the three possible terms
	\[
		-t^{2nj_1 + (3n+1)j_2-n}, \quad -t^{(2n+1)j_1+1+3nj_2}, \quad +t^{(2n+1)j_1+1 + (3n+1)j_2-n}
	\]
	coincide with $\pm t^{2n^2+k}$.
	
	(1) For the first term, the equation $2nj_1 + (3n+1)j_2 - n = 2n^2 + k$ implies that $j_2 \equiv k \pmod{n}$. Since $1 \le j_2 \le n-1$, we have $j_2 = k$. Then $j_1 = n + (1-3k)/2 \ge 0$, so such a pair $(j_1, j_2)$ exists only when $k$ is odd.
	
	(2) For the second term, the equation $(2n+1)j_1 +1 + 3nj_2 = 2n^2+k$ implies that $j_1 \equiv k-1 \pmod{n}$. Since $0 \le j_1 < n$ and thus $j_1 = k-1$. Then $j_2 = 2(n+1-k)/3$. By our assumption $n \equiv 1 \pmod{3}$, such a pair $(j_1, j_2)$ exists only when $k \equiv 2\pmod{3}$.
	
	(3) For the last term, similarly, we obtain $j_1+j_2+1 \equiv k \pmod{n}$. Since $0 \le j_1 < n$ and $1 \le j_2 < n$, we have $j_1 + j_2 + 1 = k$ or $n+k$. The case $j_1 + j_2 + 1 = k$ gives $j_1 = 3k-2n-4 < 0$, while the case $j_1 + j_2 + 1 = n+k$ gives $j_1 = n+3k-3 \ge n$. In either case, this contradicts $0 \le j_1 < n$, so the third term does not appear for any $1 \le k \le (2n+1)/3$.
	
	Therefore, the coefficient of $t^{2n^2+k}$ for $1 \le k \le (2n+1)/3$ is given by
	\[
		\begin{cases}
			0 &\text{if } k \equiv 0, 4 \pmod{6},\\
			-1 &\text{if } k \equiv 1,2,3 \pmod{6},\\
			-2 &\text{if } k \equiv 5 \pmod{6}.
		\end{cases}
	\]
	For instance, when $n \equiv 1 \pmod{9}$, the consecutive $(2n+1)/3$ coefficients considered above are all non-positive, and among them, $(4n+5)/9$ coefficients are negative. Hence, by taking $n$ sufficiently large, we see that the number of consecutive negative coefficients can be made arbitrarily large.
\end{proof}

\begin{example}
	For the modular braid $\mathbb{M} = \mathbb{M}((LLR)^4(LR)^4R) = \mathbb{L}(4,4; 5,8)$, we have
	\begin{align*}
		\Delta_{\widehat{\mathbb{M}}}(t) &= t^{68}-t^{67}+t^{60}-t^{59}+t^{56}-t^{55}+t^{52}-t^{51}+t^{48}-2 t^{46}+t^{45}+2 t^{44}-2 t^{43}+t^{40}\\
		&\qquad -t^{38}-t^{37}+3 t^{36}-t^{35}-t^{34}-t^{33}+3 t^{32}-t^{31}-t^{30}\\
		&\qquad +t^{28}-2 t^{25}+2 t^{24}+t^{23}-2 t^{22}+t^{20}-t^{17}+t^{16}-t^{13}+t^{12}-t^9+t^8-t+1.
	\end{align*}
	In the above proof, we considered the terms $t^{2n^2+k}$ with $1 \le k \le (2n+1)/3$. In this case, the three consecutive terms $t^{33}, t^{34}, t^{35}$ have coefficients all equal $-1$.
\end{example}

\bibliographystyle{amsalpha}
\bibliography{References} 

\end{document}